\documentclass[letterpaper, 10 pt, conference]{ieeeconf}  % Comment this line out
\usepackage{balance}
\usepackage{color}
\usepackage{subfigure}
\usepackage{caption}
\usepackage{cite}
\usepackage{empheq}
\usepackage{mathrsfs}
\usepackage{mleftright}
\usepackage{bbm}

\usepackage{mathtools}

\makeatletter
\let\proof\@undefined                        % undefine \proof
\let\endproof\@undefined                  % undefine \endproof
\makeatother
\usepackage{graphicx,amssymb,amstext,amsmath,amsthm}

\usepackage{multirow}
\usepackage{stfloats}

\newtheorem{prop}{Proposition} % this could go into the preamble

\newtheorem{thm}{Theorem}
	
\newtheorem{lem}{Lemma}
\newtheorem{defn}{Definition}

\newtheorem{rem}{Remark}

\newtheorem{prob}{Problem}

\usepackage{algorithm}
\usepackage[noend]{algpseudocode}
\usepackage{pifont}

\errorcontextlines\maxdimen

\makeatletter
    \newcommand*{\algrule}[1][\algorithmicindent]{\makebox[#1][l]{\hspace*{.5em}\thealgruleextra\vrule height \thealgruleheight depth \thealgruledepth}}%
\newcommand*{\thealgruleextra}{}
\newcommand*{\thealgruleheight}{.75\baselineskip}
\newcommand*{\thealgruledepth}{.25\baselineskip}

\newcount\ALG@printindent@tempcnta
\def\ALG@printindent{%
    \ifnum \theALG@nested>0% is there anything to print
        \ifx\ALG@text\ALG@x@notext% is this an end group without any text?
        \else
            \unskip
            \addvspace{-1pt}% FUDGE to make the rules line up
            \ALG@printindent@tempcnta=1
            \loop
                \algrule[\csname ALG@ind@\the\ALG@printindent@tempcnta\endcsname]%
                \advance \ALG@printindent@tempcnta 1
            \ifnum \ALG@printindent@tempcnta<\numexpr\theALG@nested+1\relax% can't do <=, so add one to RHS and use < instead
            \repeat
        \fi
    \fi
    }%
\usepackage{etoolbox}
\patchcmd{\ALG@doentity}{\noindent\hskip\ALG@tlm}{\ALG@printindent}{}{\errmessage{failed to patch}}
\makeatother

\newbox\statebox
\newcommand{\myState}[1]{%
    \setbox\statebox=\vbox{#1}%
    \edef\thealgruleheight{\dimexpr \the\ht\statebox+1pt\relax}%
    \edef\thealgruledepth{\dimexpr \the\dp\statebox+1pt\relax}%
    \ifdim\thealgruleheight<.75\baselineskip
        \def\thealgruleheight{\dimexpr .75\baselineskip+1pt\relax}%
    \fi
    \ifdim\thealgruledepth<.25\baselineskip
        \def\thealgruledepth{\dimexpr .25\baselineskip+1pt\relax}%
    \fi
    \State #1%
    \def\thealgruleheight{\dimexpr .75\baselineskip+1pt\relax}%
    \def\thealgruledepth{\dimexpr .25\baselineskip+1pt\relax}%
}
\IEEEoverridecommandlockouts

\begin{document}

\title{\LARGE \bf Weak Adaptive Submodularity and Group-Based Active Diagnosis with Applications to State Estimation with Persistent Sensor Faults}

% You will get a Paper-ID when submitting a pdf file to the conference system
%\author{Author Names Omitted for Anonymous Review. Paper-ID Sze Zheng Yong}
\author{%
Sze Zheng Yong$^{\,\rm a}$ \qquad Lingyun Gao$^{\,\rm b}$  \qquad  Necmiye Ozay$^{\,\rm b}$
%Sze Zheng Yong$^{\,\rm a}$ \quad Lingyun Gao$^{\, \rm a}$  \quad  Necmiye Ozay$^{\,\rm a}$
\thanks{This work was supported by an Early Career Faculty grant from NASA's Space Technology Research Grants Program and DARPA grant N66001-14-1-4045 and was done at the University of Michigan, Ann Arbor.}
\thanks{$^{\rm a}$ School for Engineering in Matter, Transport and Energy, Arizona State University, Tempe, AZ, USA. (email: {\tt szyong@asu.edu})}
\thanks{$^{\rm b}\, $Department of Electrical Engineering and Computer Science, University of Michigan, Ann Arbor, MI, USA. (e-mail: {\tt \{glingyun,necmiye\}@umich.edu})}\vspace*{-0.5cm}
%\thanks{$^{\rm a}$ S.Z Yong, L. Gao and N. Ozay are with the Department of Electrical Engineering and Computer Science, University of Michigan, Ann Arbor, MI, USA (e-mail: \{szyong,glingyun,necmiye\}@umich.edu).}
}

\maketitle
\thispagestyle{empty}
\pagestyle{empty}

\begin{abstract}                          % Abstract of not more than 200 words.
In this paper, we consider adaptive decision-making problems for stochastic state estimation with partial observations. First, we introduce the concept of \emph{weak adaptive submodularity}, a generalization of adaptive submodularity, which has found great success in solving challenging adaptive state estimation problems. Then, for the problem of active diagnosis, i.e., discrete state estimation via active sensing, we show that an adaptive greedy policy has a near-optimal performance guarantee when the reward function possesses this property. %Next, we consider group-based active diagnosis, which arises in applications such as medical diagnosis and state estimation with persistent sensor faults. 
We further show that the reward function for  group-based active diagnosis, which arises in applications such as medical diagnosis and state estimation with persistent sensor faults, is also weakly adaptive submodular. %; hence an adaptive greedy policy solves this problem with guaranteed near-optimal performance. 
Finally, in experiments of state estimation for an aircraft electrical system with persistent sensor faults, we observe that an adaptive greedy policy performs equally well as an exhaustive search.
\end{abstract}

\vspace{-0.025cm}
\section{Introduction}
\vspace{-0.025cm}
The operation and control of complex cyber-physical systems often requires the choice of a sequence of decisions %with uncertain outcomes 
under partial observability. This problem of adaptive decision-making appears in various applications, both in the context of stochastic control (e.g., in robot navigation \cite{kaelbling1998} and reinforcement learning \cite{sutton1998}) and stochastic state estimation (e.g., in  information gathering in robotics\cite{singh2007,javdani2014}, fault diagnosis in nuclear plants \cite{santoso1999} and sensor placement \cite{golovin2011,debouk2002}). In particular, the research area of active diagnosis and active learning, i.e., the problem of discrete object or state  identification through a minimal number of sequential actions, has many real-world applications in medical diagnosis and emergency response \cite{jaakkola1998,bellala2012}, electrical systems diagnosis \cite{maillet2013}, scheduling \cite{kosaraju1999}, 
etc.  Hence, progress in this direction can have a significant impact on a broad range of problems in robotics, cyber-physical systems and artificial intelligence. 
%The operation and control of complex cyber-physical systems typically requires the adaptive choice of a sequence of decisions with uncertain outcomes under partial observability. This problem of adaptive decision-making appears in various applications such as information gathering in robotics\cite{singh2007,javdani2014}, fault diagnosis in nuclear plants \cite{santoso1999}, power delivery systems \cite{yongli2006}, sensor placement problems \cite{golovin2011}, computer vision \cite{swain1993}, etc. In particular, the research area of active diagnosis and active learning, i.e., the problem of object or state or hypothesis identification through a minimal number of actions or queries, has many real-world applications in medical diagnosis and emergency response \cite{jaakkola1998,bellala2012}, state estimation \cite{maillet2013}, job scheduling \cite{kosaraju1999}, etc.  Hence, progress in this direction can have a significant impact on a broad range of problems in robotics, cyber-physical systems and artificial intelligence. %Yet stochastic optimization problems are known to be

\emph{Literature Review.} Finding optimal policies for general partially observable stochastic estimation problems is often an intractable combinatorial optimization problem. Hence, %for certain problems where a subset from a finite set of actions are to be chosen, 
researchers, e.g., \cite{summers2014,tzoumas2016} in the context of actuator and sensor placement, have appealed to a diminishing returns property known as submodularity, which plays a similar role in combinatorial optimization as convexity in continuous optimization. %\yong{Bla bla la la la bla Bla Bla bla la la la bla Bla Bla bla la la la bla Bla bla la la la bla Bla bla la la la bla Bla bla la la la bla Bla bla la la la bla Bla bla la la la bla}\cite{summers2014,tzoumas2016} 
Moreover, a recent paper \cite{golovin2011} introduced the notion of adaptive submodularity for set functions that extends submodularity to the adaptive setting, i.e., when actions are sequentially chosen based on observations, and provided %.introduced a particular class of set functions, called adaptive submodular functions, for which 
near-optimal performance guarantees for adaptive greedy policies. %This approach has been successfully applied in various artificial intelligence applications (cf. \cite{golovin2011} for a list) and also in discrete state estimation via active sensing \cite{maillet2013}. 
However, there are many applications in which greedy policies perform very well even when the objective function in question is not adaptive submodular. This was observed in \cite{das2011} in the non-adaptive setting, where ``near-submodular" set functions are shown to also lead to near-optimal performance guarantees. %for the dictionary selection problem, 
%and the authors introduced the notion of submodularity ratio to characterize the nearness to submodularity of a function. 
The paper \cite{kusner2014} hinted that this will hold in the adaptive case, but  without providing formal justification.  Thus, there is still a need for a rigorous definition of ``near-submodular'' set functions in the adaptive setting, and for the derivation of performance guarantees for adaptive decision-making problems with such functions. 
%Finding optimal adaptive policies for general partially observable stochastic decision-making problems is often intractable. Hence, a relatively recent paper \cite{golovin2011} introduced a particular class of set functions, called adaptive submodular functions, for which near-optimal guarantees can be derived when used in conjunction with  adaptive greedy policies. This approach has been successfully applied in various artificial intelligence applications (cf. \cite{golovin2011} for a list) and also in discrete state estimation via active sensing \cite{maillet2013}. However, there are many other applications in which greedy policies perform extraordinarily well when the function in question is not adaptive submodular. This has been observed in \cite{das2011} in the non-adaptive setting %for the dictionary selection problem, 
%and the authors introduced the notion of submodularity ratio to characterize the nearness to submodularity of a function. It was hinted that this ``near-submodular" notion also extends to the adaptive case in \cite{kusner2014} without formal justification.  Thus, there is still a need for a rigorous definition of this ``near-submodular'' property for set functions in the adaptive setting, and performance guarantees for adaptive decision-making problems with such functions. % are still lacking. % at present.

For the problem of active diagnosis (a.k.a. adaptive stochastic maximization), where actions or queries are selected sequentially to maximize the
accuracy of diagnosis with a given budget of actions, \cite{golovin2011} showed that the reward function for this problem is adaptive monotone submodular, when outcomes are not corrupted by noise or faults. Hence, an adaptive greedy policy yields the best polynomial-time approximation algorithm \cite{golovin2011}. %\cite{feige1998}.  
%has been studied in \cite{golovin2011} in the noiseless case (i.e., when outcomes are not corrupted by noise or faults). By showing that that the objective function for this problem is adaptive monotone submodular, they showed that an adaptive greedy policy yields the best polynomial-time approximation algorithm \cite{golovin2011,feige1998}.
%can be seen as a generalization of the well-known Maximum Coverage Problem in computer science and operations research to the adaptive setting. In the noiseless case (i.e., when outcomes are not corrupted by noise or faults), \cite{golovin2011} showed that the objective function for this problem is adaptive monotone submodular, hence an adaptive greedy policy yields the best polynomial-time approximation algorithm \cite{feige1998}.
However, the case when the outcomes are corrupted by noise or faults is less
well studied. The papers \cite{bellala2013,zheng2012} considered this problem when the fault is stochastic/non-persistent, i.e., the case when repeated measurements may yield different outcomes, but the proposed algorithms have no provable performance guarantees. %competitiveness with the optimal policy. 
And to our best knowledge, the active diagnosis problem with persistent noise or faults has not been considered, although this has been investigated for the complementary problem known as active learning in \cite{bellala2012,golovin2010}. In these approaches, proxy set functions were found to indirectly prove the near-optimality of the original problem. However, corresponding specialized algorithms were needed and it is unclear how these apply to our active diagnosis problem. In fact, these papers \cite{bellala2012,golovin2010} considered a more general problem of group-based active learning, also studied in \cite{javdani2014}, of which active learning with persistent noise is a special case. The objective of group-based active learning is to learn the group to which an object or objects belong, as opposed to the objects themselves with a minimum number of actions or queries. On the other hand, the problem of group-based active diagnosis has, to the extent of our knowledge, not been investigated, despite its potentially broad applicability.

\emph{Contributions.} To close the gap between non-adaptive and adaptive settings, we rigorously define the concept of \emph{weak adaptive submodularity} for set functions, as a generalization of (strong) adaptive submodularity, and prove that an adaptive greedy policy %, approximate or exact, 
yields competitive near-optimal performance guarantees for the active diagnosis problem, when the objective function is adaptive monotone and weakly adaptive submodular. We observe that the adaptive submodularity factor that characterizes nearness to (strong) adaptive submodularity affects the performance bound in the same way as when approximate greedy policies are employed.

Then, we consider the \emph{group-based active diagnosis} problem and show that the reward function corresponding to group identification is both adaptive monotone and weakly adaptive submodular; thus %Since this is an instance of the active diagnosis problem, 
we also obtain competitive performance guarantees for this problem %e group-based active diagnosis 
when using a simple greedy policy without the need for proxy set functions nor specialized algorithms. We believe that our approach is novel and valuable especially for its simplicity. Moreover, the adaptive submodularity factor also offers an alternative explanation for the deterioration of the performance guarantees that was observed in group-based active learning when proxy set functions are introduced (cf. \cite{bellala2012,golovin2010,javdani2014}).

Furthermore, we demonstrate that the problem of \emph{active diagnosis with persistent noise or fault} is equivalent to the group-based active diagnosis problem; hence, once again, we have provably competitive performance guarantees for the adaptive greedy policy and do not require proxy set functions and algorithms. In addition, our general formulation allows for state-dependent sensor noise or fault, and also complex faults, %beyond just `flipping the bit' (i.e., giving the opposite outcome), 
such as faults that give the same outcome regardless of the system states, or faults that always give the opposite outcome. Our experiments on aircraft electrical systems show that the simple adaptive greedy policy performs just as well as a brute-force policy that tries all possible actions, %(thus, violating the given budget on the number of actions or queries), 
while only having a polynomial-time computational complexity.

\section{Motivation: Aircraft Electrical System with Persistent Sensor Faults}
\vspace{-0.05cm}
We are motivated by the discrete state estimation problem of an aircraft electrical system via active sensing, first studied in \cite{maillet2013} (see Figure \ref{fig:simple} for an example of a simple circuit and the readers are referred to \cite{maillet2013} for a detailed description of the electrical circuits). As more systems become more dependent on electric power systems, this problem is crucial to the safety of these systems. In the above work, the sensors are assumed to be \emph{healthy} or faultless. But in reality, sensors can be faulty, which means that they can provide incorrect information about the unknown discrete state (i.e., operating condition or status) of the electrical components. The faults can be either stochastic or persistent, and in this paper, we are mainly interested in persistent faults. One such fault is when a faulty sensor persistently gives the opposite reading and another is when a faulty sensor always gives a constant reading (hence sometimes correctly and others incorrectly). 

Despite persistent faults, our goal is similar to that in \cite{maillet2013}, i.e., to design a policy that can adaptively estimate the discrete state of the circuit by taking ``actions" (i.e., opening or closing controllable contactors) and observing the sensor measurements. Note that we are only interested in estimating the unknown discrete state of the system and not the true failure mode of the sensors. Hence, if we group all possible sensor failure modes that correspond to each state, our problem of interest is one of \emph{group-based active diagnosis}, i.e., we aim to adaptively identify or estimate which group (i.e., state) is compatible with our sensor measurements. However, the reward function for this problem is not adaptive submodular, which motivates the need for the concept of weak adaptive submodularity, discussed next.

\begin{figure}[!t]
\centering
\includegraphics[width=0.25\textwidth]{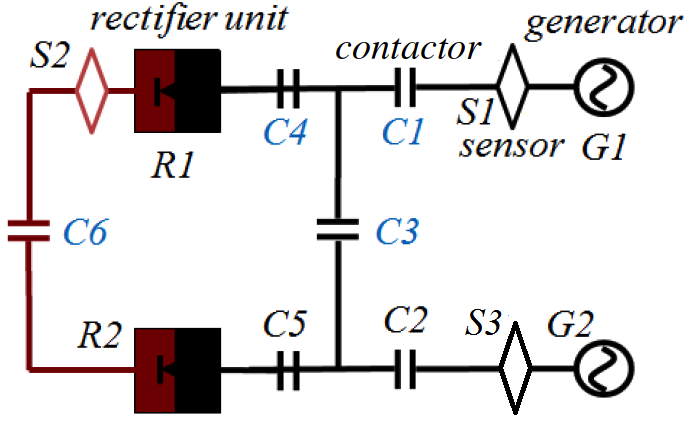} %\vspace{-0.075cm}

\caption{A single-line diagram of a simple circuit with AC components (in black) and DC components (in red).}\vspace{-0.05cm}
\label{fig:simple}
\end{figure}

%\vspace{-0.05cm}
\section{Weak Adaptive Submodularity}
\vspace{-0.05cm}
In this section, we propose a generalization of the concept of adaptive submodularity and show that adaptive greedy policies achieve near-optimal performance for objective functions with such a property. These results will be useful for group-based active diagnosis, considered in the next section.

\vspace{-0.05cm}
\subsection{Mathematical Preliminaries and Definitions}%Definitions and Problem Statement}
\vspace{-0.05cm}
%\begin{prob}[Adaptive Stochastic Maximization] \label{prob:1}
%The objective of the Adaptive Stochastic Maximization problem is to find a policy $\pi^*$ such that
%\begin{align}
%\begin{array}{l}
%\pi^* \in \displaystyle\arg \max_{\pi} \mathbb{E}[f(\tilde{\mathcal{V}}(\pi,\mathbf{X}),\mathbf{X})]\\
%\mathrm{subject \ to\quad } |\tilde{\mathcal{V}}(\pi,\mathbf{X})| \leq k, \forall \mathbf{x} \in \boldsymbol{\mathcal{X}},
%\end{array}
%\end{align}
%with expectation taken with respect to $\mathbb{P}[\mathbf{x}]$.
%\end{prob}

We begin by introducing notations for the adaptive decision-making problem we consider. %Our decisions involve the adaptive choice of an action (or a query) from 
Given a finite set of actions (or queries), $v\in\mathcal{V}$, our objective is to adaptively select actions to estimate a fixed but hidden state, $\mathbf{x}_0\in\boldsymbol{\mathcal{X}}$. We consider a Bayesian approach and model the state as a random variable, $\mathbf{X}$, with a given prior probability distribution $\mathbb{P}[\mathbf{x}]$ on $\boldsymbol{\mathcal{X}}$. When we take an action $v$, we observe the (sensor) measurement or outcome $y\in\mathcal{Y}$, which we will use to sequentially make a decision about the next action. We represent the pairs of actions $\lbrace v_1, \hdots, v_t\rbrace$ and observed outcomes $\lbrace y_1, \hdots, y_t\rbrace$ up to time $t$ by a partial realization $\psi_t=\lbrace(v_i, y_i)\rbrace_{i\in\lbrace 1, ..., t\rbrace}$. Given two partial realizations $\psi_t$ and $\psi_{t'}$, we call $\psi_t$ a subrealization of $\psi_{t'}$ if $\psi_t\subseteq\psi_{t'}$.

To evaluate ``progress" in the decision-making process, 
we define an objective function $f : 2^{\mathcal{V}}\times \boldsymbol{\mathcal{X}} \to \mathbb{R}_{\geq 0} $ that maps the set of actions $A \subseteq 2^\mathcal{V}$ under state $\mathbf{x}_0$ to reward $f(A, \mathbf{x}_0)$. %\yong{Since the set of observations, $Y_A$, is deterministic, we will abuse the notation and use $f(A, \mathbf{x}_0)$ in place of $f(A,Y_A, \mathbf{x}_0)$ for conciseness.}  
A policy $\pi$ is a function from a set of partial realizations to actions, which specifies the action to choose next  after observing  $\psi_t$, i.e., $v_{t+1}=\pi(\psi_t)$, 
 and  $\tilde{\mathcal{V}}(\pi, \mathbf{x}_0)\subseteq 2^\mathcal{V}$ denotes the set of all actions under policy $\pi$. Randomized policies that specifies the distribution on actions are also allowed.

Next, we provide definitions specific to set functions in \cite{golovin2011}, followed by the definition of weak adaptive submodularity.

\vspace{-0.1cm}
\begin{defn}[Conditional Expected Marginal Benefit] \label{def:marginGain}
Given an objective function $f$, an action $v \in \mathcal{V}$ and a partial realization $\psi_t$, the \emph{conditional expected marginal benefit} of $v$ conditioned on having observed $\psi_t$ is defined as
\begin{align} \label{eq:margin}
\Delta(v|\psi_t) \triangleq \mathbb{E}[f(v_{1:t} \cup \{v\}, \mathbf{X} )-f(v_{1:t}, \mathbf{X} )|\psi_t],
\end{align}
with the expectation taken with respect to $\mathbb{P}[\mathbf{x}|\psi_t]$. Similarly, the conditional expected marginal benefit of a policy $\pi$ is
\begin{align} 
\Delta(\pi|\psi_t) \triangleq \mathbb{E}[f(v_{1:t} \cup \tilde{\mathcal{V}}(\pi, \mathbf{x}_0), \mathbf{X} )-f(v_{1:t}, \mathbf{X} )|\psi_t].
\end{align}
\end{defn}

The following definition of adaptive monotonicity has the interpretation that the conditional expected marginal benefit of any action $v$ is non-negative.
\begin{defn}[Adaptive Monotonicity] \label{def:monotone}
A function $f:2^{\mathcal{V}}\times \boldsymbol{\mathcal{X}} \to \mathbb{R}_{\geq 0}$ is \emph{adaptive monotone} with respect to distribution $\mathbb{P}[\mathbf{x}]$ if  for all $v \in \mathcal{V}$ and $\psi_t$ with $\mathbb{P}[\psi_t]>0$, $\Delta (v|\psi_t) \geq 0.$
\end{defn}

Next, we define a new property that has the interpretation that 
%Moreover, if 
the conditional expected marginal benefit of any fixed action (or query) $v$ does not increase ``too much" as more actions are performed and their measurements are observed. %, then following property holds.

\begin{defn}[$\zeta$-Weak Adaptive Submodularity] \label{def:submodular}
A function $f:2^{\mathcal{V}}\times \boldsymbol{\mathcal{X}} \to \mathbb{R}_{\geq 0}$ is \emph{$\zeta$-weakly adaptive submodular} with respect to distribution $\mathbb{P}[\mathbf{x}]$  if for all $\psi_t,\psi_{t'}$ such that $\psi_t$ is a subrealization of $\psi_{t'}$, i.e., $\psi_t \subseteq \psi_{t'}$, and for all $v \in \mathcal{V}\setminus v_{1:t'}$, %we have
$$ \Delta (v|\psi_{t'}) \leq \zeta \Delta (v|\psi_t),$$
for some constant $\zeta \geq 1$ that we refer to as the adaptive submodularity factor. Additionally, let $\zeta^*\leq \zeta$ be the smallest such $\zeta$, referred to as the best adaptive submodularity factor.
\end{defn}

From the above definition, it is clear that for any $\zeta \geq \zeta'$, if $f$ is $\zeta'$-weakly adaptive submodular, then $f$ is also $\zeta$-weakly adaptive submodular. Further, (strongly) adaptive submodular functions $f$ (with $\zeta^*=1$) are also $\zeta$-weakly adaptive submodular for any $\zeta \geq 1$; hence, this is a generalization of (strong) adaptive submodularity. In fact, any set function is $\zeta$-weakly adaptive submodular with some (possibly infinite) $\zeta$. %Note also that $\zeta^*$ has some relation to the reciprocal of `submodularity ratio' defined in \cite{das2011} for the non-adaptive setting, and 
Note that our definition differs from approximate adaptive submodularity that was defined in \cite{kusner2014} without justification.

%\subsection{Adaptive Greedy Policy} % for Adaptive Submodular Functions}
%
%For reward functions with the approximate adaptive submodularity property, the adaptive greedy algorithm, i.e., a strategy that selects the action that maximizes the conditional expected marginal benefit, conditioned on all previous actions and observations:
%\begin{align} \label{eq:greedy1}
%v_{t+1} &\in \displaystyle\arg\max_{v \in \mathcal{V}} \Delta (v|\psi_t), 
%\end{align}
%with $\Delta(v|\psi_t)$ as given in \eqref{eq:margin}, can be shown to provide very nice performance guarantees, described in the next subsections. 
%%In fact, for the problems that we consider (cf. Problems \ref{prob:1} and \ref{prob:2}), it was shown that no polynomial time approximation algorithm can outperform the adaptive greedy strategy in \cite{feige1998} and \cite[Lemma A.14]{golovin2011}, respectively.
%% unless $P = NP$ for the adaptive stochastic maximization problem and unless $NP \subseteq DTIME(|\boldsymbol{\mathcal{X}}|^{O(\log \log |\boldsymbol{\mathcal{X}}|)})$ \cite[Lemma A.14]{golovin2011} for the adaptive stochastic minimum cost cover problem.
%
%\subsubsection{Incorporating Non-Uniform Action Cost}
%
%\subsubsection{$\alpha$-Approximation Greedy Policies}

\subsection{Adaptive Greedy Policies}

Weak adaptive submodularity resembles a diminishing returns property for policies, which suggests that an adaptive greedy policy can provide nice performance guarantees.
Thus, we consider an adaptive greedy policy that, at each iteration, myopically follows the greedy heuristic of maximizing its expected gain given its current observations: 
\begin{align*}
\pi^{greedy}(\psi_t)\triangleq v_{t+1} \in \arg\max_{v \in \mathcal{V}} \Delta(v|\psi_t).
\end{align*}
Such greedy strategies do not in general produce an optimal solution when compared to clairvoyant algorithms, but are oftentimes near-optimal in practice and  have the advantage of having a polynomial-time computational complexity.

%Several extensions of the adaptive greedy policy (cf. \cite{golovin2011}) can also be easily handled as follows:
%\subsubsection{Non-Uniform Action Costs}
%
%As with problems with (strong) adaptive submodularity, the adaptive greedy algorithm can naturally handle non-uniform action costs by simply modifying its selection criterion to be
%\begin{align*}
%v_{t+1} \in \arg\max_{v \in \mathcal{V}} \frac{\Delta(v|\psi_t)}{c(v)}.
%\end{align*}
%
%\subsubsection{Approximate Adaptive Greedy Policies}
%Moreover, if finding an action that maximizes the above is computationally intractable, we can instead consider an $\alpha$-approximate greedy policy that finds an action $v_{t+1}$ such that
%\begin{align*}
%\frac{\Delta(v_{t+1}|\psi_t)}{c(v_{t+1})} \geq \frac{1}{\alpha} \max_{v \in \mathcal{V}} \frac{\Delta(v|\psi_t)}{c(v)}.
%\end{align*}
%For clarity of exposition, in the following, we will only focus on the uniform cost and exact adaptive greedy policy case. We will defer the analysis with costs and approximate adaptive greedy policies to Appendix \ref{app:submodular}.

\subsection{Performance Guarantees for Active Diagnosis} %Adaptive Stochastic Maximization}

We consider an important class of problems as follows: %For both reward functions, i.e., $f_{sub}(v_{1:t},\mathbf{x})$ for the group identification problem with subgroups and $f_{sup}(v_{1:t},\mathbf{x})$ for the group identification problem with supergroups or regions, we consider the following problems.

%\begin{prob}[Adaptive Stochastic Maximization] \label{prob:1}
\begin{prob}[Active Diagnosis] \label{prob:1}
%The objective of the Adaptive Stochastic Maximization Problem is to 
Given a reward function $f$, find a policy $\pi^*$ with a budget of $k$ actions such that
\begin{align}
\begin{array}{l}
\pi^* \in \displaystyle\arg \max_{\pi} \mathbb{E}[f(\tilde{\mathcal{V}}(\pi,\mathbf{X}),\mathbf{X})]\\
\mathrm{subject \ to\quad } |\tilde{\mathcal{V}}(\pi,\mathbf{x})| \leq k, \forall \mathbf{x} \in \boldsymbol{\mathcal{X}},
\end{array}
\end{align}
with expectation taken with respect to $\mathbb{P}[\mathbf{x}]$ and $\tilde{\mathcal{V}}(\pi,\mathbf{x}) \subseteq 2^\mathcal{V}$ is the set of all actions under policy $\pi$ when the state is $\mathbf{x}$.
\end{prob}

It turns out that for this class of problems, the weak adaptive submodularity property offers a very nice performance guarantee when used in conjunction with adaptive greedy policies, which we will prove in Appendix \ref{app:submodular}.
%\begin{lem}[Adaptive Data Dependent Bound] \label{lem:bound}
%Suppose we have observed partial realizations $\psi_t$ after selecting $v_{1:t}$ and let $\pi^*$ be any policy such that $|\tilde{\mathcal{V}}(\pi^*,\mathbf{x}_0)|\leq k$ for all $\mathbf{x}_0$. Then for adaptive monotone and $\zeta$-adaptive near-submodular $f:2^\mathcal{V} \times \boldsymbol{\mathcal{X}} \to \mathbb{R}_{\geq 0}$, we have
%\begin{align*}
%\Delta(\pi^*|\psi_t) \leq \zeta \max_{\mathcal{A}\subseteq \mathcal{V},|\mathcal{A}|\leq k} \sum_{v \in \mathcal{A}} \Delta(v|\psi_t).
%\end{align*}
%\end{lem}

%Note that this is possiblty what Kusner was trying to write but did not, and $\zeta$-approximate adaptive submodularity is a sufficient condition for this, which is really what is needed for showing the performance guarantee.

\begin{thm} \label{thm:greedy1}
Fix any $\zeta \geq 1$. Let the greedy policy $\pi_\ell^{greedy}$ be run for $\ell$ iterations (so that it select $\ell$ actions), and $\pi^*_k$ be any policy selecting at most $k$ actions for any realization $\mathbf{x}$. Then for adaptive monotone and $\zeta$-weakly adaptive submodular $f$, %we have %, and $\pi$ is a greedy policy, then for all policies $\pi^*$ and positive integers $\ell$ and $k$, we have
\begin{align*}
f_{avg}(\pi_\ell^{greedy}) > (1-e^{-\ell/\zeta k}) f_{avg}(\pi^*_k),
\end{align*}
where $f_{avg} (\pi) \triangleq \mathbb{E}[f(\tilde{\mathcal{V}}(\pi,\mathbf{X}),\mathbf{X})]$ is the expected reward of the policy $\pi$ with respect to the distribution $\mathbb{P}[\mathbf{x}]$.
\end{thm}

A corollary from the above theorem is that our greedy policy that selects $k$ actions obtains at least $(1-e^{-1/\zeta})$ of the value of the optimal policy that selects $k$ actions. Note that if $\zeta$ is large, then the performance bound can be poor. %In fact, it was shown in \cite{feige1998} that no polynomial time approximation algorithm can outperform the adaptive greedy strategy. % and \cite[Lemma A.14]{golovin2011}, respectively

%\begin{rem}
%Lemma \ref{lem:bound} generalizes the equivalent definition of (non-adaptive) submodularity given by $f(T)\leq f(S) + \sum_{v \in T \setminus S} \Delta(v|S)$ for all $S\subseteq T \subseteq \boldsymbol{\mathcal{X}}$ to the adaptive setting with $\zeta \geq 1$. Further, $\zeta^*$ is the reciprocal of the \emph{submodularity ratio} defined for the non-adaptive case in \cite{das2011}.
%\end{rem}

\begin{rem}
The adaptive submodularity factor $\zeta$ affects the performance guarantee in Theorem \ref{thm:greedy1}  %(see also Theorem \ref{eq:greedyGuarantee}) 
in the exact way as the approximation factor of $\alpha$-approximate greedy policies. Just as approximate greedy policies suggest that greedy policies are robust to incorrect priors $\mathbb{P}[\mathbf{x}]$ by a factor $\alpha$ (cf. \cite[Section 4.3]{golovin2011} for a discussion), the $\zeta$-weak adaptive submodularity property is ``robust" to the deviation from (strong) adaptive submodularity by a factor $\zeta$.
\end{rem}

\section{Near-Optimal Group-Based Active Diagnosis}
\subsection{Problem Setup}

The objective of group-based active diagnosis is not to determine the unknown object but rather the group to which the object belongs via active sensing. For instance, in medical diagnosis, one is more interested to diagnose the disease that a patient has without necessarily identifying all its symptoms, while in movie recommendation systems, it may be of more interest to identify the genre that a user prefers as opposed to the exact scene preferences. As also noted in \cite{bellala2012}, the group-based diagnosis problem cannot be simply reduced to an active diagnosis problem with groups as ``meta-states", since the objects within a group do not generally yield the same observations. In our motivating example, our goal is to primarily estimate the hidden system state with no emphasis on finding the hidden persistent sensor faults. Further, under the hypotheses of different sensor failure modes for the same system state, the sensor measurements would be different. 

\subsubsection{Mathematical formulation} \label{sec:math}

In this section, we further introduce notations that are specific to the group-based active diagnosis problem. We will only consider non-overlapping groups in this paper; thus, to make the connection to active diagnosis with persistent faults clear, we will refer to groups as (system) states $x \in \mathcal{X}$ and objects in the groups as (sensor) modes $q \in \mathcal{Q}$. 
%Since it is clear that any group-based active diagnosis with non-overlapping decision regions (assumed in this paper) can be described with pairs of states and modes, we will introduce our mathematical notations for this problem in terms of system states and sensor modes. 
 We assume that the true state $x_0\in\mathcal{X}$ and the true mode $q_0\in\mathcal{Q}$ are fixed but unknown. As before, we will consider a Bayesian approach 
and model the state $X$ and mode $Q$ as random variables with a given prior joint probability distribution $\mathbb{P}[x,q]$ on $\mathcal{X} \times \mathcal{Q}$. Note that the modes need not be the same for each state. $\mathcal{Q}$ can be taken as the union of all modes and if a particular mode $q$ does not belong to a state $x$, then $\mathbb{P}[x,q]=0$.

%is pre-determined according to component properties. At the beginning of each estimation process, the system is in an unknown state $x_0\in\Omega$ and $q_0\in\mathcal{Q}$. We assume $x_0$ and $q_0$ will not change during the process.

%When we take an action (or make a query) $v\in\mathcal{V}$, we observe the (sensor) measurement/outcome $y\in\mathcal{Y}$. 
Given the true pair of state $x_0$ and mode $q_0$, $y=\mu(v,x_0,q_0)\in\mathcal{Y}$ is the unique outcome of performing action $v$. %The pairs of actions $\lbrace v_1, \hdots, v_t\rbrace$ and observed outcomes $\lbrace y_1, \hdots, y_t\rbrace$ are represented by the partial realization $\psi_t=\lbrace(v_i, y_i)\rbrace_{i\in\lbrace 1, ..., t\rbrace}$. Given two partial realizations $\psi_t$ and $\psi_{t'}$, we call $\psi_t$ a subrealization of $\psi_{t'}$ if $\psi_t\subseteq\psi_{t'}$. 
Further, we define $D(y,v,q)$ as the set of states $x \in \mathcal{X}$ that gives the same outcome $y$ under the action $v$ if $q$ were the true mode. %and $(x,q)$ is the true state and mode pair. 
We then define for each iteration $t$, $S_{t,q}$ as the set of all compatible states with the hypothesis that $q$ is the true mode, i.e., the set of all states that produce the same set of outcomes $\lbrace\mu(v_1, x_0, q_0), ..., \mu(v_t, x_0, q_0)\rbrace$  under the set of actions $\lbrace v_1, ..., v_t\rbrace$, %assuming sensor state is $q$. We denote $S_{t,q}$ to be the shorthand for $h(v_{1:t}, x_0, q_0, q)$. At each step $t$, we take a new action $v'$, there exists a recursive relation:
%\begin{equation}
%h(v_{1:t}\cup\lbrace v'\rbrace,x_0,q_0,q)=h(v_{1:t},x_0,q_0,q)\cap D(\mu(v', x_0,q_0), v',q)
%\end{equation}
which can also be written as
\begin{equation}
S_{t,q} = \cap_{i\in\lbrace1,...,t\rbrace}D(\mu(v_i,x_0,q_0),v_i,q).
\end{equation}
Since only intersections are taken, the order of actions $v_i$ does not matter.

\subsubsection{Problem statement}
The objective of group-based active diagnosis %in the context of state estimation with persistent sensor faults 
is to eliminate as many of the states as possible that are incompatible with observed outcomes, i.e., the uncertainty of $X$ represented by $\mathbb{P}[x]$. For the group-based active diagnosis problem, this means that at any step $t$, we cannot eliminate any state $x$ that is compatible with the hypothesis of one or more modes, i.e., $x \in \cup_{q\in \mathcal{Q}} S_{t,q}$. Thus, we define a reward function $f: 2^{\mathcal{V}}\times \mathcal{X}\times \mathcal{Q}\rightarrow\mathbb{R}_{\geq 0}$ as follows %which maps the set of actions $A\subseteq\mathcal{V}$ under state $x_0,q_0$ to reward $f(A,x_0,q_0)$.
\begin{align} \label{eq:reward}
f(v_{1:t},\mathbf{x})=f(v_{1:t},x,q)=1-\hspace{-0.1cm}\sum_{x \in \bigcup_{q \in \mathcal{Q}} S_{t,q}} \hspace{-0.2cm}\mathbb{P}[x],
%\begin{array}{rl}
%%f(v_{1:t},x,q)&=-\mathbb{P}[\cup_{x \in \mathcal{X}} \{x: \exists q \in \mathcal{Q} : x \in S_{t,q} \}]\\
%&f(v_{1:t},\mathbf{x})=f(v_{1:t},x,q)\\
%%&=1-\mathbb{P}[\cup_{q \in \mathcal{Q}} \{\mathcal{X}: \mathcal{X} \cap S_{t,q} \neq \emptyset\}]+ \mathbb{P}[x]\\
%&=1+\mathbb{P}[x]-\sum_{x \in \cup_{q \in \mathcal{Q}} S_{t,q}} \mathbb{P}[x],
%\end{array} \hspace{-0.1cm}
\end{align}
%where $\mathbf{x}=\begin{bmatrix} x^T & q^T \end{bmatrix}^T$.
%and we also define $S_t \triangleq \cup_{q \in \mathcal{Q}} S_{t,q}$.
%\begin{rem}
with $\mathbf{x}=(x,q)$. Note that only the final term in \eqref{eq:reward} is essential for the greedy algorithms that we propose. If $\mathbb{P}[x]$ is uniformly distributed on $\mathcal{X}$, then this term is proportional to the size of $\cup_{q\in\mathcal{Q}}S_{t,q}$. The constant is added such that the reward function is non-negative, and is equal to 0 when no action has been taken, i.e., $f(\emptyset,\mathbf{x})=0$. 

%A policy $\pi$ is a function of the observed partial realizations $\psi_t$ to action $v$, i.e., $v_{t+1}=\pi(\psi_t)$, while $\tilde{\mathcal{V}}(\pi,x_0,q_0) \subseteq \mathcal{V}$ is the set of all actions performed under policy $\pi$, when the state and mode pair is $(x_0,q_0)$.
%To assess the uncertainty of the state estimation, \textbf{we define an objective function $f: 2^{\mathcal{V}\times\mathcal{Y}\times\mathcal{Q}}\times\Omega\rightarrow\mathbb{R}_+$ which maps the set of actions $A\subseteq\mathcal{V}$ under state $x_0,q_0$ to reward $f(A,x_0,q_0)$.} We define the strategy as $\pi(\psi_t)$. It chooses the next action $v_{t+1}$ by observing $\psi_t$. $\tilde{\mathcal{V}}(\pi, x_0,q_0)\subseteq\mathcal{V}$ is the set of all the actions by using strategy $\pi$. Generally, $\tilde{\mathcal{V}}(\pi, x_0,q_0)\neq\mathcal{V}$.

We also assume that we are allocated a budget on the maximum number of actions taken, $k\ll |2^\mathcal{V}|$. To achieve the goal, we want to design a policy $\pi^*$ that finds the ``best expected estimate" of the state $x$. 
%This $\pi^*$ should satisfy:
%Formulated in this manner, the active group diagnosis problem is equivalent to the Maximum Coverage Problem in Problem \ref{prob:1}.
\vspace{-0.05cm}
\begin{prob}[Group-Based Active Diagnosis] \label{prob:2}
Find a policy $\pi^*$ that solves Problem \ref{prob:1} with $f$ given in \eqref{eq:reward}.
%The objective of Group-Based Active Diagnosis is equivalent to Adaptive Stochastic Maximization in Problem \ref{prob:1} with $f$ given in \eqref{eq:reward}.
\end{prob}
\vspace{-0.05cm}
\subsubsection{Active diagnosis with persistent faults (special case)}
To consider this problem as a group-based active  diagnosis problem (see also \cite{bellala2012,golovin2010} for a similar characterization), we create groups of objects where each group corresponds to a state and the objects within the group are copies of the state for each possible mode of sensor faults. For example, we consider a system with $|\mathcal{X}|$ states, each denoted by $x \in \mathcal{X}$. If we have $m$ sensors that can be faulty and there are $p$ ways a sensor fault can manifest itself, then we have $(p+1)^m$ possible combinations of functioning and faulty sensor modes, denoted by $q \in \mathcal{Q}$, with cardinality $|\mathcal{Q}|=(p+1)^m$. Next, we define an object as a pair of a state and a mode, denoted as $\mathbf{x}=(x,q) \in \boldsymbol{\mathcal{X}}=\mathcal{X}\times \mathcal{Q}$. %, and the set of objects is formed by the Cartesian product of the set of system states and the set of modes. 
In other words, for each group/state $x$, all pairs of $\{x,q\}_{q\in \mathcal{Q}}$ form the objects corresponding to this group.

Next, to introduce the types of sensor faults we consider, we define the function $\overline{\mu}(v,x_0)$ as the unique outcome under action $v$ and true state $x_0$ when all sensors are healthy and the function $y'=\kappa_q(y)$ as the corruption of the `healthy' sensor outcome $y$ to the `faulty' sensor outcome $y'$ when the failure mode is $q$. Hence, $\mu(v,x,q)=\kappa_q(\overline{\mu}(v,x))$.  %And at each step $t$. %\textbf{ At each step $t$, the probability distribution $\mathbb{P}[x]$ and $\mathbb{P}[q]$ can be updated by conditioning it on $\psi_t$ and get $\mathbb{P}[x\mid\psi_t]$ and $\mathbb{P}[q\mid\psi_t]$.}
%In order to achieve out goal of eliminating the invalid states, 
Further, we define $\overline{D}(y,v)$ with $y=\overline{\mu}(v,x)$ as the set of states $x \in \mathcal{X}$ that gives the same `healthy' sensor measurement $y$ under the action $v$ when $x$ is the actual system state. % of the system and the sensors are all functioning or healthy. 
Thus, $D(y,v,q)=\overline{D}(\kappa_q^{-1}(y),v)$, where the inverse function $\kappa_q^{-1}$ is such that $y'=\kappa_q^{-1}(y)$ implies that $y=\kappa_q(y')$.

Note that the fault function $\kappa_q$ is surjective, whereas $\kappa_q^{-1}$ is an injective partial function, as can be seen with the types of sensor faults we consider in this paper. First, a Type 1 fault is when a faulty sensor persistently outputs the opposite outcome. Another type of sensor fault is when the faulty sensor always gives a constant reading, e.g., `0' or `1' (Type 2 fault). %, where `0' is faulty and `1' is healthy. 
For a single sensor with a Type 1 fault, $y=\kappa_q^{-1}(\kappa_q(y))$, whereas for one with a Type 2 fault, $y \in \kappa_q^{-1}(\kappa_q(y))$ or $\kappa_q^{-1}(\kappa_q(y))=\emptyset$. To further illustrate this, consider a simple example with three sensors: $s_1, s_2, s_3$, where only $s_2$ is faulty, i.e., $q=[1,0,1]$, and the `faulty' sensor measurement is $y=[1, 0, 1]$. %We make an assumption that generators and rectifier units only have two states: \textit{unhealthy} (represented by 0) or \textit{not unhealthy} (either \textit{offline} or \textit{healthy}, represented by 1). Faulty sensors have two types. Type 1 is that the faulty sensor always gives the opposite readings. Type 2 is that the faulty sensor always gives 1. We know $s_2$ is faulty and $q$ represents the sensor state when $s_2$ is faulty. 
If the fault is of Type 1, we have $y'=\kappa_q^{-1}(y)=[1,1,1]$ and $D(y,v,q)=\overline{D}(y',v)$, whereas if the fault is of Type 2 (with persistent outcome `1'), $\kappa_q^{-1}(y)=\{y',y''\}$, with $y'=[1,0,1]$ and $y''=[1,1,1]$, and $D(y,v,q)=\overline{D}(y',v)\cup \overline{D}(y'',v)$. On the other hand, if $s_2$ is faulty of Type 2 (with persistent outcome `0') and if $y_2$ should always be 1 for all $q$, then $y'=\kappa_q^{-1}(y)$ is undefined and by definition, $D(y,v,q)=\emptyset$.

%\vspace{-0.05cm}
\subsection{Adaptive Greedy Policy: Group-Based Active Diagnosis}
\vspace{-0.05cm}
The optimal policy for Problem \ref{prob:2} may only be obtained by a clairvoyant algorithm. Moreover, the complexity of planning ahead for $k$ steps scales exponentially with $k$. Hence, we resolve to find a scalable (polynomial-time) algorithm using an adaptive greedy strategy, i.e., to select the action at time $t+1$ that maximizes the expected one-step reward given the available information $\psi_t$, denoted  by $\Delta (v|\psi_t)$ and given in Definition \ref{def:marginGain}, %:
%\begin{align} \label{eq:greedy1}
%\begin{array}{rl}
%v_{t+1} &\in \displaystyle\arg\max_{v \in \mathcal{V}} \Delta (v|\psi_t)\\
%&=\displaystyle\arg\max_{v \in \mathcal{V}} \mathbb{E}[f(v_{1:t}\cup \{v\},\mathbf{X})-f(v_{1:t},\mathbf{X}) | \psi_t]
%\end{array}
%\end{align} 
where the expectation is taken with respect to $\mathbb{P}[x,q|\psi_t]$. This probability measure $\mathbb{P}[x,q|\psi_t]$ on the set $S_{t,q}$ for each $q \in \mathcal{Q}$ can be obtained via Bayes' rule:
\begin{align}
\mathbb{P}[x,q|\psi_t]&= \frac{\mathbb{P}[\psi_t |x,q] \mathbb{P}[x,q]}{\mathbb{P}[\psi_t]}=\begin{cases}
\frac{\mathbb{P}[x,q]}{\mathbb{P}[\psi_t]} \hspace{0.35cm} \forall x \in S_{t,q},\\
0 \hspace{.85cm} otherwise.
\end{cases} \hspace{-0.2cm}\label{eq:condP}
\end{align}
Note that we substituted $\mathbb{P}[\psi_t|x,q]=1$ if $x \in S_{t,q}$ because the measurement process is deterministic, and $\mathbb{P}[\psi_t|x,q]=0$ otherwise. Moreover, since $\sum_{q \in \mathcal{Q}} \sum_{x \in S_{t,q}} \mathbb{P}[x,q|\psi_t]=1$, the normalization term $\mathbb{P}[\psi]$ can be computed as
\begin{align}\label{eq:normalization}
\mathbb{P}[\psi_t]=\sum_{q \in \mathcal{Q}} \sum_{x \in S_{t,q}} \mathbb{P}[x,q].
\end{align}\vspace{-0.25cm}

Thus, at each iteration or step $t$,  the policy consists of choosing the next best action $v_{t+1}$ that maximizes the gain in uncertainty reduction as defined in Definition \ref{def:marginGain}, which can be rewritten for the reward function in \eqref{eq:reward} as
%Thus, for the subgroup reward function, $f_{sub}(v_{1:t}\cup \{v\},\mathbf{x})$, the greedy strategy in \eqref{eq:greedy1} can be rewritten as
\begin{align*} 
\begin{array}{ll}
 v_{t+1} &\in 
\displaystyle\arg\max_{v \in \mathcal{V}}  \frac{1}{\mathbb{P}[\psi_t]}  \displaystyle\sum_{q \in \mathcal{Q}} \sum_{x \in S_{t,q}} \mathbb{P}[x,q]\\ %\mathbb{P}[x]\mathbb{P}[q|x]\\
&  \quad (\hspace{-0.1cm}\displaystyle\sum_{\tilde{x} \in \cup_{\tilde{q} \in \mathcal{Q}} S_{t,\tilde{q}}} \mathbb{P}[\tilde{x}]-\hspace*{-0.2cm}\sum_{\tilde{x} \in \cup_{\tilde{q} \in \mathcal{Q}} S_{t,\tilde{q}} \cap D(\mu(v,x,q),v,\tilde{q})} \hspace{-0.1cm}\mathbb{P}[\tilde{x}] \, ),%\\
%&=\displaystyle\arg\max_{v \in \mathcal{V}}  \frac{1}{\mathbb{P}[\psi_t]} (\displaystyle\sum_{\tilde{x} \in \cup_{q \in \mathcal{Q}} S_{t,q}} \mathbb{P}[\tilde{x}] \\
%&  \ -\displaystyle\sum_{q \in \mathcal{Q}} \sum_{x \in S_{t,q}}  \mathbb{P}[x,q] \hspace{-0.1cm}\sum_{\tilde{x} \in \cup_{\tilde{q} \in \mathcal{Q}} S_{t,\tilde{q}} \cap D(\mu(v,x,q),v)} \hspace{-0.1cm} \mathbb{P}[\tilde{x}] \ ),
 \end{array}
\end{align*} 
where $\mathbb{P}(\psi_t)=\displaystyle\sum_{q \in \mathcal{Q}} \sum_{x \in S_{t,q}} \mathbb{P}[x,q]$. Since the first term does not depend on the choice of action $v$, the above expression can be equivalently written as
\begin{align} \label{eq:greedysub}
\hspace{-0.2cm}v_{t+1}&\in 
\displaystyle\arg\min_{v \in \mathcal{V}} \displaystyle\sum_{q \in \mathcal{Q}} \sum_{x \in S_{t,q}} \mathbb{P}[x,q] \hspace{-0.5cm}\sum_{\substack{\tilde{x} \in \bigcup_{\tilde{q} \in \mathcal{Q}} (S_{t,\tilde{q}} \\ \qquad \cap D(\mu(v,x,q),v,\tilde{q}))}} \hspace{-0.3cm} \mathbb{P}[\tilde{x}].
\end{align}

In addition, a less computationally expensive algorithm can be obtained by further transforming \eqref{eq:greedysub} to
\begin{align} \label{eq:greedy}
\hspace{-0.2cm}v_{t+1}&\in 
\displaystyle\arg\min_{v \in \mathcal{V}} \sum_{y \in \mathcal{Y}} \hspace{-0.15cm}\sum_{\substack{\tilde{x} \in \bigcup_{\tilde{q} \in \mathcal{Q}} (S_{t,\tilde{q}} \\ \qquad \cap D(y,v,\tilde{q}))}} \hspace{-0.35cm} \mathbb{P}[\tilde{x}] \sum_{q \in \mathcal{Q}} \hspace{-0.35cm}\sum_{\substack{x \in S_{t,q} \cap\\ \qquad  D(y,v,q)}} \hspace{-0.3cm}\mathbb{P}[x,q] \hspace{-0.2cm}
\end{align}
by using the fact that the sets $\{D(y,v,q)\ |\ y \in \mathcal{Y}\}$ form a partition of $\mathcal{X}$ as will be described in more detail in the proof of Lemma \ref{lem:Delta} in Appendix \ref{app:proofs}. %This version of the adaptive greedy policy is provided in Algorithm \ref{algorithm:greedy}.

\vspace{-0.05cm}
\subsection{Near-Optimality Performance Guarantees}
\vspace{-0.05cm}
To derive performance guarantees for the group-based active diagnosis problem with the reward function $f(v_{1:t},x,q)$ in \eqref{eq:reward}, we first show that $f$ satisfies the following (which will be proven in Appendix \ref{app:proofs}).
\begin{prop}[Adaptive Monotonicity] \label{prop:monotone}
The reward function $f(v_{1:t},x,q)$ in \eqref{eq:reward} is adaptive monotone.
\end{prop}
%\begin{proof}
%This follows directly since $\Delta (v|\psi_t)$ can be expressed with the function $b$ in \eqref{eq:b}.
%\end{proof}

\begin{prop}[$\zeta$-Weak Adaptive Submodularity] \label{prop:submodular}
The reward function $f(v_{1:t},x,q)$ in \eqref{eq:reward} is $\zeta$-weakly adaptive submodular with 
\begin{align}
1\leq \zeta \leq \overline{\zeta}  \leq \frac{|\mathcal{Q}|}{\min_{\{x \in \mathcal{X}, q \in \mathcal{Q}: \mathbb{P}[x,q]>0\}}\mathbb{P}[x,q]}, \label{eq:zeta}
\end{align} 
where $\zeta$ and ${\overline{\zeta}}$ are defined as
\begin{align*}
{\zeta} \triangleq \max_{\substack{v_{1:k} \in \mathcal{V}^k, \, v \in \mathcal{V},\\y \in \mathcal{Y}, \, t=1,\hdots,k}} \frac{\sum_{x \in \bigcup_{{q} \in \mathcal{Q}} S_{t,{q}}\cap D(y,v,{q})} \mathbb{P}[x]}{\sum_{q\in \mathcal{Q}} \sum_{x\in S_{t,q}\cap D(y,v,q)} \mathbb{P}[x,q]},\\
{\overline{\zeta}} \triangleq \max_{\substack{v_{1:k} \in \mathcal{V}^k, \, v \in \mathcal{V},\\y \in \mathcal{Y}, \, t=1,\hdots,k}} \frac{\sum_{q \in \mathcal{Q}}\sum_{x \in S_{t,{q}}\cap D(y,v,q)} \mathbb{P}[x]}{\sum_{q\in \mathcal{Q}} \sum_{x\in S_{t,q} \cap D(y,v,q)} \mathbb{P}[x,q]}.
\end{align*}
\end{prop}
%\begin{proof}
%This follows directly from Lemma \ref{lem:b}.
%\end{proof}
Both ${\zeta}$ and ${\overline{\zeta}}$ can be computed algorithmically, but since we are interested in a smaller $\zeta$, we only provide an algorithm for computing ${\zeta}$ in Algorithm \ref{algorithm:zeta}. ${\overline{\zeta}}$ is of interest because it can be easily derived analytically for some special cases. Specifically, for a group-based active diagnosis problem with uniform $\mathbb{P}[x,q]$ for each $x \in \mathcal{X}$, which is the case for active sensing with no informative priors on the sensor faults, we have $\mathbb{P}[x]=|\mathcal{Q}| \mathbb{P}[x,q]$. In this case, ${\overline{\zeta}}=|\mathcal{Q}|$. Note that these factors are only required for finding the performance guarantee, and are not needed for the adaptive greedy strategy we propose.

\begin{algorithm}[t]
\caption{Computation of $\zeta$}
\label{algorithm:zeta}
\begin{algorithmic}[1]
\State Initialize $\zeta=0$ and $S_{0,q}=\mathcal{X}$ for all $q \in \mathcal{Q}$
\For{each action sequence $v_{1:k} \in \mathcal{V}^k$}
\For{$t \in \{1,\hdots,k\}$}
\State Perform $v_t$ and measure $y_t=\mu(v_t,x_0,q_0)$
\State $S_{t,q}=S_{t-1,q} \cap D(y_t,v_t,q)$ for all $q \in \mathcal{Q}$
\For{each action $v \in \mathcal{V}$}
\For{each output $y\in \mathcal{Y}$}
\State $\overline{S}_{t,q}=S_{t,q} \cap D(y,v,q)$ for all $q \in \mathcal{Q}$
\State $\overline{S}_t=\bigcup_{q\in \mathcal{Q}}\overline{S}_{t,q}$
\State $\zeta_{t,v,y}=\frac{\sum_{x \in \overline{S}_{t}} \mathbb{P}[x]}{\sum_{q\in \mathcal{Q}} \sum_{x\in \overline{S}_{t,q}} \mathbb{P}[x,q]}$
\State $\zeta \leftarrow \max\{\zeta,\zeta_{t,v,y}\}$
\EndFor
\EndFor
%\State $\psi_t=\psi_{t-1}\cup(v_0,y_0)$
\EndFor
\EndFor
\State \Return{$\zeta$}
%\State Compare the values of $f$
\end{algorithmic}
\end{algorithm}

Since the group-based active diagnosis problem (Problem \ref{prob:2}) is equivalent to the active diagnosis problem  with an adaptive monotone and weakly adaptive submodular reward function, by Propositions \ref{prop:monotone} and \ref{prop:submodular}, as well as Theorem \ref{thm:greedy1}, we have the following performance guarantee. 
%\subsubsection{Adaptive Stochastic Maximization}

\begin{thm}
For any true state ${x} \in {\mathcal{X}}$ and true mode $q \in \mathcal{Q}$, the adaptive greedy policy $\pi^{greedy}_k$  guarantees that 
$$f_{avg}(\pi^{greedy}_{k}) > (1- e^{-1/\zeta}) f_{avg}(\pi^*_k)$$
for the group-based active diagnosis problem with the reward function in \eqref{eq:reward}, 
where $f_{avg}(\pi^*_k)$ is what can be achieved in $k$ steps by any other policy, including the optimal policy and $\zeta$ is given in \eqref{eq:zeta}. %$1\leq \zeta \leq \frac{|\mathcal{Q}|}{\min\limits_{x \in \mathcal{X}, q \in \mathcal{Q}: \mathbb{P}[x,q]>0}\mathbb{P}[x,q]} $.%\\
%\phantom{bl} Moreover, no polynomial time approximation algorithm can outperform the adaptive greedy strategy. 
\end{thm}\vspace{-0.15cm}

Hence, in the absence of informative priors on the sensor faults (i.e., when ${\overline{\zeta}}=|\mathcal{Q}|$), our adaptive greedy policy that selects $k$ actions obtains at least $(1-e^{-1/|\mathcal{Q}|})$ of the value of the optimal strategy that selects $k$ actions. 
To our best knowledge, this result is novel. Group identification for the complementary problem of adaptive stochastic minimum cost cover, a.k.a. Bayesian active learning has been considered in \cite{golovin2010,bellala2012,javdani2014}, but not the group-based active diagnosis problem we consider in this paper. Interestingly, the performance bounds given in \cite{golovin2010,bellala2012} have an additional factor of 2, which is equal to $\zeta=|\mathcal{Q}|=2$ since they had exactly 2 modes for each persistent sensor fault, while \cite{javdani2014} also has a factor that resembles $\zeta$.

%In fact, our approach may also be applicable for the complementary problem, which is an ongoing subject of our research. The advantage of this new result would be the justification for using the adaptive greedy policy without the need for proxies and corresponding new algorithms as in \cite{golovin2010} and \cite{bellala2012}.

% and \cite[Lemma A.14]{golovin2011}, respectively

%\section{Implementation Details}
%\subsection{Lazy Evaluations}
\vspace{-0.05cm}
\section{Examples}

%To see how the adaptive greedy policy in Algorithm \ref{algorithm:greedy} performs, 
We tested the adaptive greedy policy %in Algorithm \ref{algorithm:greedy}  
on two electric power system circuits---small and large circuits. For these experiments, the state $x$ is uniformly distributed on $\mathcal{X}$, and for any sensor measurement that is prone to faults, $q_i$ for all $i=1,\hdots,|\mathcal{Q}|$, the conditional probabilities that faults of Type 1 and Type 2 (with outcome `stuck' at `1') persistently occur are $\mathbb{P}[q^1_{i}|x]=0.2$ and $\mathbb{P}[q^2_{i}|x]=0.4$, respectively. %If sensor $i$ is not prone to faults, then $\mathbb{P}[q_{(i)}|x]=1$. 
The conditional probability for each state and mode pair is then $\mathbb{P}[q| x]=\prod_{i=1}^{|\mathcal{Q}|} \mathbb{P}[q_{i}|x]$. For simplicity, we only allow two states for the system components, i.e., they are either \textit{healthy} or \textit{faulty}. Correspondingly, the sensor readings can only take two values, i.e., \textit{proper voltage} or \textit{improper voltage}. We do not assume an initial action $v_0$, so we initialize our experiments with $S_{0,q}=\mathcal{X}$ for each $q$. %We assume there are two fail types that cannot overlap. Type 1 gives the opposite readings of the correct ones. Type 2 gives the constant readings of \textit{not improper voltage}.

For many practical cases, it is not possible to reduce the size of the compatible states to one because the number and positioning of sensors are often limited, and more so when the sensors can be faulty. Thus, we compare the performance of our adaptive greedy policy with the results of a brute force policy that exhaustively takes every action $v \in \mathcal{V}$ (thus, violating the given budget on the number of actions). The brute force policy represents the best achievable outcome, which we use as a benchmark. %The brute force strategy takes every action $v\in\mathcal{V}$. 
%If the results of the two strategies are the same, if means the greedy strategy gets the best results that the estimation can achieve. But brute force strategy is usually not practical because a large circuit will take long time to take all the actions. Our simulation is done without initial configuration. The test methodology is described below. We also need to decide the state of sensors and which fault type they have, and simulate on different cases.
Both policies are run in Python on an \textit{Intel\textregistered \ Core$^{TM}$} i5-4300U 64-bit CPU@1.90GHz 2.50GHz with 8.00GB RAM, and the experiments are performed for all possible system states $x_0 \in \mathcal{X}$ and sensor modes $q_0 \in \mathcal{Q}$.

%\begin{algorithm}
%\caption{Test methodology}
%\label{CHalgorithm}
%\begin{algorithmic}[1]
%\State Set the initial state of sensors $q_0\in\mathcal{Q}$
%\For{each state $x_0\in\Omega$}
%\State Set the circuit in the state $x_0$
%\State Run the greedy strategy
%\State Record computation time and the value of $f$
%\State Run the brute force strategy
%\State Record the value of $f$
%\EndFor
%\State Compare the values of $f$
%\end{algorithmic}
%\end{algorithm}
\vspace{-0.125cm}
\subsection{Tests on a small-size circuit}
\label{Tests on a small-size circuit}
\vspace{-0.05cm}
We first tested the proposed adaptive greedy policy on a small circuit with 13 components, as shown in Figure \ref{fig:simple}. By using four controllable contactors $\{C1, C3, C4, C6\}$ (thus, $|\mathcal{V}|=2^4=16$ actions) and observing the measurements of sensors $\{S1, S2, S3\}$, which may be faulty,  the objective is to estimate the unknown states of the components $\{G1, G2, R1, R2, C2, C5\}$. %by Using four controllable contactors $(C1, C3, C4, C6)$ By observing the readings of sensors $(S1, S2, S3)$ under certain sensor states. Based on the controllable contactors, we can take mostly $\mid\mathcal{V}\mid=2^4=16$ actions. 

\textit{Simulation result}: The adaptive greedy policy that takes 6 actions performs equally well as the brute force policy that takes all 16 actions, i.e., the values of the objective function $f$ are equal for all possible states $x_0$ and failure modes $q_0$.

\subsubsection{Average execution time} We recorded the average execution times to find the next best action, and the results are shown in Figure \ref{fig:execution} for two representative examples for different true sensor modes $q$. %We do not show all the graphs under all sensor states. Fig. 2 gives the result when all the sensors are healthy. Fig. 3 is the case when $S2$ is faulty with type 1. Fig. 4 gives the result when $S1, S2$ are faulty with type 1 and $S3$ are faulty with type 2. 
It takes about 25 milliseconds to compute the next best action when all the sensors are healthy, and we observed that the average execution time increases linearly in the number of modes $|\mathcal{Q}|=2^m$, i.e., exponentially with the number of fault-prone sensors $m$. %With one more faulty sensor, either with type 1 or type 2, the execution time doubles.

%\begin{figure}[H]
%\centering
%\includegraphics[width=0.5\textwidth]{Figures/2.png} 
%
%\caption{Histogram of execution time for the greedy strategy with all sensors healthy}
%\label{Histogram of execution time for the greedy strategy with all sensors healthy}
%\end{figure}

\begin{figure}[!tp]
\centering
\subfigure[$S2$ is fault-prone of Type 1 and $S1, S3$ are healthy. ]{
  \includegraphics[width=0.231\textwidth,trim= 0mm 0mm 2mm 3mm,clip]{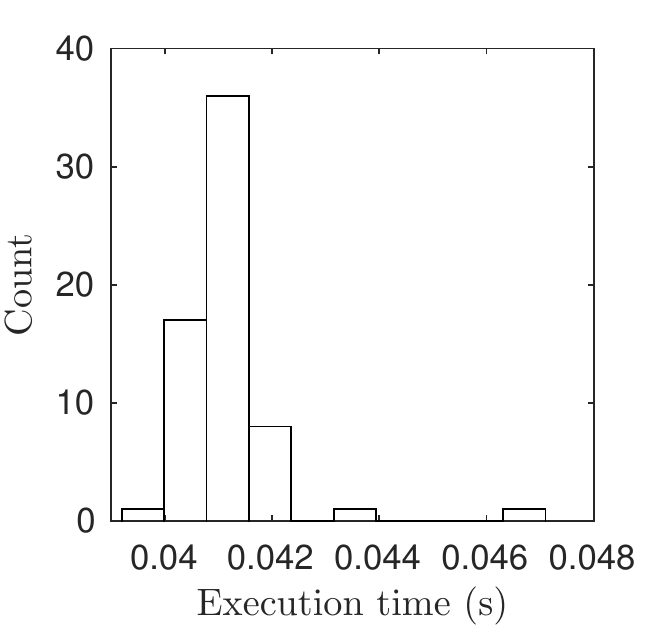} \label{fig:S2}}
  \subfigure[$S1, S2$ are fault-prone of Type 1 and $S3$ of Type 2.]{
  \includegraphics[width=0.218\textwidth,trim= 0mm 0mm 5mm 3mm,clip]{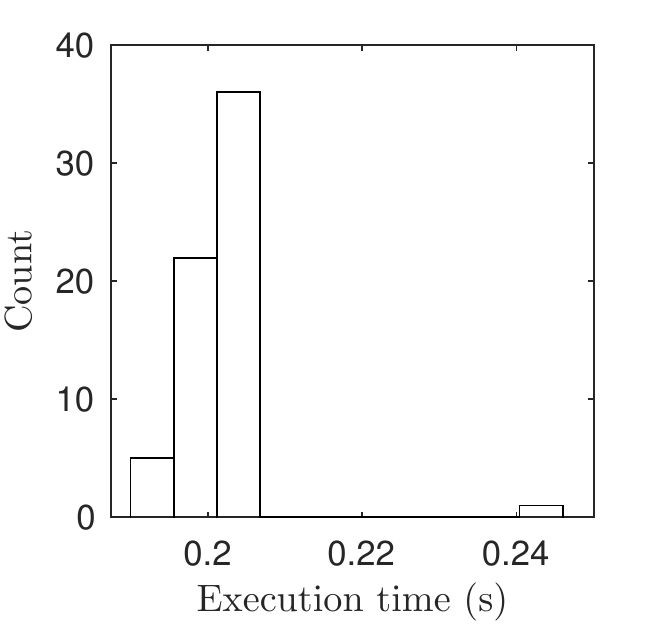} \label{fig:all}}
  
  \caption{Histogram of execution times (averaged over 6 actions) of the adaptive greedy policy for all true states $x_0$.} \label{fig:execution}
\end{figure}

\begin{figure}[!tp]
\centering
%\subfigure[$S2$ is fault-prone of Type 1 and $S1, S3$ are healthy. ]
{
  \includegraphics[width=0.49\textwidth,trim= 7mm 0mm 12.5mm 0mm,clip]{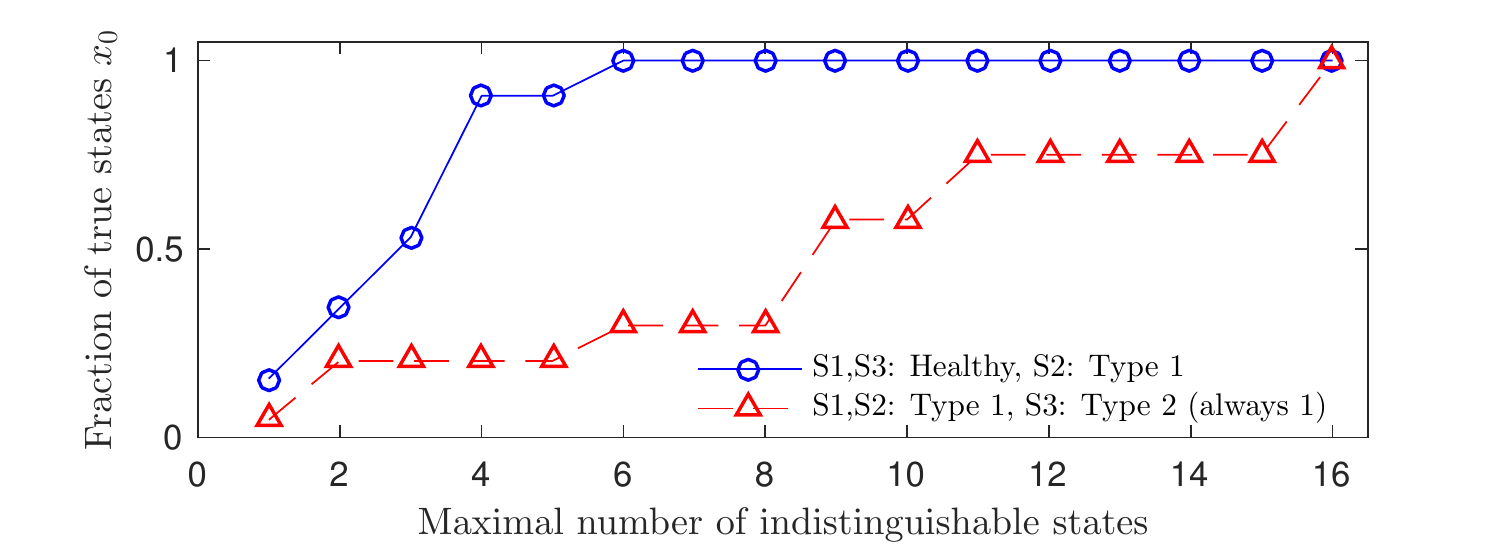}}

%  \subfigure[$S1, S2$ are fault-prone of Type 1 and $S3$ of Type 2.]{
%  \includegraphics[width=0.45\textwidth,trim= 0mm 10mm 2mm 15mm,clip]{Figures/6.png} }
  
  \caption{Fraction of true states $x_0$ with its corresponding maximal number indistinguishable states after 6 actions.} \label{fig:reward}
\end{figure}

\subsubsection{Maximal number of indistinguishable states} Figure \ref{fig:reward} compares the maximal numbers of indistinguishable states (corresponding to the minimal values of the reward function $f$) after 6 actions for all $x_0$ and two illustrative examples of $q_0$. The ordinate shows the fraction of system states $x_0$ (out of $2^6=64$) that obtains the number of indistinguishable states, $|\bigcup_{q\in \mathcal{Q}} S_{6,q}|$, that is less than or equal to the maximal number (on the abscissa). Since the adaptive greedy strategy performs equally well as the brute force method,  we omit the plots for the brute force policy that lie exactly on top of the plots for the adaptive greedy policy.  Moreover, comparing Figures \ref{fig:S2} and \ref{fig:all}, we observe that the maximal number of indistinguishable states is larger when more sensors are faulty, as expected. %since there will be naturally be more uncertainty. %A point at coordinates $(n\%, m)$ represents that among $m\%$ of the cases, the final result is eliminated to $n\%$ of all the possibilities. 
Furthermore, for most system states $x_0$, the best reward is already obtained by our adaptive greedy policy within 2 actions.

%\begin{figure}[H]
%\centering
%\includegraphics[width=0.5\textwidth]{Figures/5.png} 
%
%\caption{Histogram of execution time for the greedy strategy with all sensors healthy}
%\end{figure}

%\begin{figure}[H]
%\centering
%\includegraphics[width=0.5\textwidth]{Figures/6.png} 
%
%\caption{Histogram of execution time for the greedy strategy with $S1, S2$ faulty of type 1 and $S3$ faulty of type 2}
%\end{figure}
\vspace{-0.1cm}
\subsection{Tests on a larger circuit}
\label{Tests on a larger circuit}
\vspace{-0.05cm}
We also tested our greedy policy on %a (simplified version) of 
a larger circuit with more sensors and components (cf. %Figure \ref{fig:large} and 
\cite[Section V-B]{maillet2013} for details), which is more representative of aircraft power distribution systems. %shows the circuit topology before and after simplification. The simplified process is described in Quentin's paper. 
%It takes 49.1 seconds to do the offline computation, and is ten times longer than the small circuit, which only takes 4.76 seconds. 
When all sensors are healthy, the average execution time for obtaining the next best action is about 0.22 seconds. As before, the average execution time increases linearly with the number of sensor modes $|\mathcal{Q}|$. Similarly, the performance of our adaptive greedy policy with $6$ actions is as good as the brute force policy with $32$ actions.

%\begin{figure}[t]
%\centering
%\includegraphics[width=0.425\textwidth]{Figures/7.png} 
%\caption{A larger circuit with AC and DC components \cite{maillet2013}. with controllable contacters depicted in blue. \label{fig:large}}
%\end{figure}

%\begin{figure}[H]
%\centering
%\includegraphics[width=0.45\textwidth]{Figures/8.png} 
%\caption{Histogram of execution time for the greedy strategy with $S1, S2$ faulty of type 1 and $S3$ faulty of type 2 \label{fig:large}}
%\end{figure}
%\vspace{-0.05cm}
\section{Conclusions and Future Directions}
%\vspace{-0.05cm}
In this paper, we considered stochastic state estimation problems with partial observations via active sensing. First, we introduced a property for set functions, called \emph{weak adaptive submodularity}, that generalizes the concept of adaptive submodularity. Then, for the active diagnosis problem, we showed that adaptive greedy policies are near-optimal when the reward function is adaptive monotone and weakly adaptive submodular. Next, we considered the group-based active diagnosis problem, for which a special case is the active diagnosis problem with persistent sensor noise or faults, and proved that the group-based reward function is weakly adaptive submodular; %and this problem is simply an instance of the Adaptive Stochastic Maximization problem; 
hence the group-based active diagnosis problem can be solved using a simple adaptive greedy policy with guaranteed competitive performance when compared to the optimal adaptive policy. Our state estimation experiments with aircraft electrical systems plagued by persistent sensor faults demonstrated that the adaptive greedy policy performs just as well as a brute-force policy.

Future work will consider group-based active learning for decision-making using weak adaptive submodularity with the hope of removing the need for proxy set  functions and specialized algorithms, while preserving the provable competitiveness of adaptive greedy policies. We will also consider the general active diagnosis problem where the sensor noise and faults can be persistent and/or non-persistent.\vspace{-0.05cm}

\bibliographystyle{unsrt}        % Include this if you use bibtex 
\bibliography{biblio}           % and a bib file to produce the 

\begin{thebibliography}{10}

\bibitem{kaelbling1998}
L.P. Kaelbling, M.L. Littman, and A.R. Cassandra.
\newblock Planning and acting in partially observable stochastic domains.
\newblock {\em Artificial intelligence}, 101(1):99--134, 1998.

\bibitem{sutton1998}
R.S. Sutton and A.G. Barto.
\newblock {\em Reinforcement learning: An introduction}, volume~1.
\newblock MIT Press Cambridge, 1998.

\bibitem{singh2007}
A.~Singh, A.~Krause, C.~Guestrin, W.J. Kaiser, and M.A. Batalin.
\newblock Efficient planning of informative paths for multiple robots.
\newblock In {\em IJCAI}, volume~7, pages 2204--2211, 2007.

\bibitem{javdani2014}
S.~Javdani, Y.~Chen, A.~Karbasi, A.~Krause, D.~Bagnell, and S.S. Srinivasa.
\newblock Near optimal bayesian active learning for decision making.
\newblock In {\em AISTATS}, pages 430--438, 2014.

\bibitem{santoso1999}
N.I. Santoso, C.~Darken, G.~Povh, and J.~Erdmann.
\newblock Nuclear plant fault diagnosis using probabilistic reasoning.
\newblock In {\em Power Engineering Society Summer Meeting, 1999. IEEE},
  volume~2, pages 714--719. IEEE, 1999.

\bibitem{golovin2011}
D.~Golovin and A.~Krause.
\newblock Adaptive submodularity: Theory and applications in active learning
  and stochastic optimization.
\newblock {\em Journal of Artificial Intelligence Research}, 42:427--486, 2011.

\bibitem{debouk2002}
R.~Debouk, S.~Lafortune, and D.~Teneketzis.
\newblock On an optimization problem in sensor selection.
\newblock {\em Discrete Event Dynamic Systems}, 12(4):417--445, 2002.

\bibitem{jaakkola1998}
T.S. Jaakkola and M.I. Jordan.
\newblock Variational methods and the qmr-dt database.
\newblock {\em NATO ASI Series F Computer and Systems Sciences}, 168:185--214,
  1998.

\bibitem{bellala2012}
G.~Bellala, S.K. Bhavnani, and C.~Scott.
\newblock Group-based active query selection for rapid diagnosis in
  time-critical situations.
\newblock {\em IEEE Transactions on Information Theory}, 58(1):459--478, 2012.

\bibitem{maillet2013}
Q.~Maillet, H.~Xu, N.~Ozay, and R.M. Murray.
\newblock Dynamic state estimation in distributed aircraft electric control
  systems via adaptive submodularity.
\newblock In {\em 52nd IEEE Conference on Decision and Control}, pages
  5497--5503. IEEE, 2013.

\bibitem{kosaraju1999}
S.R. Kosaraju, T.M. Przytycka, and R.~Borgstrom.
\newblock On an optimal split tree problem.
\newblock In {\em Workshop on Algorithms and Data Structures}, pages 157--168.
  Springer, 1999.

\bibitem{summers2014}
T.H. Summers and J.~Lygeros.
\newblock Optimal sensor and actuator placement in complex dynamical networks.
\newblock {\em IFAC Proceedings Volumes}, 47(3):3784--3789, 2014.

\bibitem{tzoumas2016}
V.~Tzoumas, A.~Jadbabaie, and G.J. Pappas.
\newblock Sensor placement for optimal {K}alman filtering: {F}undamental
  limits, submodularity, and algorithms.
\newblock In {\em 2016 American Control Conference}, pages 191--196, July 2016.

\bibitem{das2011}
A.~Das and D.~Kempe.
\newblock Submodular meets spectral: Greedy algorithms for subset selection,
  sparse approximation and dictionary selection.
\newblock {\em arXiv preprint arXiv:1102.3975}, 2011.

\bibitem{kusner2014}
M.J. Kusner.
\newblock Approximately adaptive submodular maximization.
\newblock In {\em NIPS workshop on Discrete Optimization and Machine Learning},
  2014.

\bibitem{bellala2013}
G.~Bellala, J.~Stanley, S.K. Bhavnani, and C.~Scott.
\newblock A rank-based approach to active diagnosis.
\newblock {\em IEEE transactions on pattern analysis and machine intelligence},
  35(9):2078--2090, 2013.

\bibitem{zheng2012}
A.X. Zheng, I.~Rish, and A.~Beygelzimer.
\newblock Efficient test selection in active diagnosis via entropy
  approximation.
\newblock {\em arXiv preprint arXiv:1207.1418}, 2012.

\bibitem{golovin2010}
D.~Golovin, A.~Krause, and D.~Ray.
\newblock Near-optimal bayesian active learning with noisy observations.
\newblock In {\em Advances in Neural Information Processing Systems}, pages
  766--774, 2010.

\end{thebibliography}
                                 % bibliography (preferred). The
                                 % correct style is generated by
                                 % Elsevier at the time of printing.

%\begin{thebibliography}{99}     % Otherwise use the  
                                 % thebibliography environment.
                                 % Insert the full references here.
                                 % See a recent issue of Automatica 
                                 % for the style.
%  \bibitem[Heritage, 1992]{Heritage:92}
%     (1992) {\it The American Heritage. 
%     Dictionary of the American Language.}
%     Houghton Mifflin Company.
%  \bibitem[Able, 1956]{Abl:56}
%     B.~C.~Able (1956). Nucleic acid content of macroscope. 
%     {\it Nature 2}, 7--9. 
%  \bibitem[Able {\em et al.}, 1954]{AbTaRu:54}   
%     B.~C. Able, R.~A. Tagg, and M.~Rush (1954).
%     Enzyme-catalyzed cellular transanimations.
%     In A.~F.~Round, editor, 
%     {\it Advances in Enzymology Vol. 2} (125--247). 
%     New York, Academic Press.
%  \bibitem[R.~Keohane, 1958]{Keo:58}
%     R.~Keohane (1958).
%     {\it Power and Interdependence: 
%     World Politics in Transition.}
%     Boston, Little, Brown \& Co.
%  \bibitem[Powers, 1985]{Pow:85}
%     T.~Powers (1985).
%     Is there a way out?
%     {\it Harpers, June 1985}, 35--47.

%\end{thebibliography}

\vspace{-0.2cm}
%\appendix
%\section{A summary of Latin grammar}    % Each appendix must have a short title.
%\section{Some Latin vocabulary}         % Sections and subsections are supported  
                                        % in the appendices.                                      
\appendix
%\section{Analysis and Proofs}

\subsection{Proof for Weak Adaptive Submodularity (Theorem \ref{thm:greedy1})} \label{app:submodular}

The following lemma will be used to prove Theorem \ref{thm:greedy1}.

\begin{lem}%[Adaptive Data Dependent Bound] 
\label{lem:bound}
Suppose we have observed partial realizations $\psi_t$ after selecting $v_{1:t}$ and let $\pi^*$ be any policy such that $|\tilde{\mathcal{V}}(\pi^*,\mathbf{x}_0)|\leq k$ for all $\mathbf{x}_0$. Then for adaptive monotone and $\zeta$-weakly adaptive submodular $f:2^\mathcal{V} \times \mathbf{\mathcal{X}} \to \mathbb{R}_{\geq 0}$, we have
\begin{align*}
\Delta(\pi^*|\psi_t) \leq \zeta \max_{\mathcal{A}\subseteq \mathcal{V},|\mathcal{A}|\leq k} \sum_{v \in \mathcal{A}} \Delta(v|\psi_t).
\end{align*}
\end{lem}
%\begin{lem}[Adaptive Data Dependent Bound with Costs] \label{lem:boundCost}
%Suppose we observed partial realizations $\psi_t$ after selecting $v_{1:t}$ and let $\pi^*$ be any policy. Then for adaptive monotone and $\zeta$-adaptive near-submodular $f:2^\mathcal{V} \times \mathbf{\mathcal{X}} \to \mathbb{R}_{\geq 0}$, %we have
%\begin{align*}
%\Delta(\pi^*|\psi_t) \leq \zeta Z \leq \zeta c(\pi^*|\psi_t) \max_{v \in \mathcal{V}} \left(\frac{\Delta(v|\psi_t)}{c(v)}\right),
%\end{align*}
%where $Z\triangleq\max\limits_w\{\sum\limits_{v \in \mathcal{V}} w_v \Delta(v|\psi_t):\sum\limits_{v\in \mathcal{V}} c(v) w_v \leq c(\pi^*|\psi_t) \ \rm{and} \ \forall v \in \mathcal{V}, 0 \leq w_v \leq 1\}$.
%\end{lem}
\begin{proof}
This result can be obtained by a slight adaptation of the proof for Lemma A.9 in \cite{golovin2011}, where we replaced $\Delta(v|\psi_{t'}) \leq \Delta(v|\psi_t)$ with $\Delta(v|\psi_{t'}) \leq \zeta\Delta(v|\psi_t)$ when $\psi_{t'}$ is a subrealization of $\psi_t$, and with uniform action cost. 
\end{proof}

%\begin{thm} \label{eq:greedyGuarantee}
%Fix any %$\alpha\geq 1$, 
%$\zeta \geq 1$ and action costs $c: \mathcal{V} \to \mathbb{R}_{> 0}$. Let the %$\alpha$-approximate greedy policy $\pi_\ell^{\alpha{\text{-}}greedy}$ 
%greedy policy $\pi_\ell^{greedy}$ be run for $\ell$ iterations (so that it select $\ell$ actions), and $\pi^*_k$ be any policy selecting at most $k$ actions for any realization $\mathbf{x}$. 
%Then for adaptive monotone and $\zeta$-weakly adaptive submodular $f$,
%\begin{align*}
%f_{avg}(\pi_\ell^{greedy}) > (1-e^{-\ell/\zeta k}) f_{avg}(\pi^*_k).
%\end{align*}
%where $f_{avg} (\pi) \triangleq \mathbb{E}[f(\tilde{\mathcal{V}}(\pi,\mathbf{X}),\mathbf{X})]$ is the expected reward of the policy $\pi$ with respect to the distribution $\mathbb{P}[\mathbf{x}]$.
%\end{thm}
\begin{proof}[Proof of Theorem \ref{thm:greedy1}]
The proof goes along the lines of the analysis in Theorem A.10 in \cite{golovin2011} but for the generalized case when $\zeta\geq 1$ and with uniform action cost. 
First, we consider a corollary of Lemma \ref{lem:bound} for any $i$, %\ref{lem:boundCost}, 
which allows us to bound
\begin{align*}
\mathbb{E}[\Delta(\pi^*_k|\psi_i)] \leq \zeta k \max_{v \in \mathcal{V}} (\Delta(v|\psi_i)),
\end{align*}
where the expectation is taken over the internal randomness of $\pi^*_k$, if there is any, and we know that the budget of actions for all $\mathbf{x}_0$ is such that $\mathbb{E}[|\tilde{\mathcal{V}}(\pi^*,\mathbf{x}_0)|]\leq k$. % by definition. 
%It then follows that 
%$$\mathbb{E}[\Delta(\pi^*_k|\psi_i)] \leq \zeta k \max_{v \in \mathcal{V}} (\Delta(v|\psi_i)).$$
Then, %for all $i=0,\hdots,\ell$, %by definition, a greedy policy, $\pi_i$ (superscript omitted for brevity) obtains 
%\begin{align*}
%\begin{array}{rl}
%f_{avg}(\pi^*_k)-f_{avg}(\pi_i) &\leq f_{avg}(\pi_i@\pi^*_k)-f_{avg}(\pi_i)\\
%&= \sum_{j=1}^k \Delta()
%\end{array}
%\end{align*}
\begin{align*} 
f_{avg}(\pi_{i+1})-f_{avg}(\pi_i)&=\max_{v \in \mathcal{V}} (\Delta(v|\psi_i)) \\
&\geq \mathbb{E}[\Delta(\pi^*_k|\psi_i)]/ \zeta k\\
&=(f_{avg}(\pi_i@\pi^*_k)-f_{avg}(\pi_i))/ \zeta k,
\end{align*}
where $\pi_i@\pi^*_k$ is a policy obtained by running $\pi_i$ to completion, and then running policy $\pi^*_k$ as if from a fresh start and ignoring the gathered information $\psi_i$. 
The first equality holds because $\pi_i$ (superscript omitted for brevity) is a greedy policy, while the final equality holds by %because
%expected marginal benefit per unit cost in a step immediately following its observation of $\psi_i$. Next, 
taking an appropriate convex combination over the random partial realizations $\Psi_i$. %, we have
%\begin{align*}
%\begin{array}{rl}
%f_{avg}(\pi_{i+1})-f_{avg}(\pi_i) %&\geq \mathbb{E}\left[\frac{1}{\alpha} \max_{v \in \mathcal{V}} \Delta(v|\Psi_i)\right]\\
% &\geq \mathbb{E}\left[\frac{ \mathbb{E}
% \left[\Delta(v|\Psi_i)\right]}{ \zeta k} \right]\\
% &=\frac{f_{avg}(\pi_i@\pi^*_k)-f_{avg}(\pi_i)}{ \zeta k},
% \end{array}
%\end{align*}
%where $\pi_i@\pi^*_k$ is a policy obtained by running $\pi_i$ to completion, and then running policy $\pi^*_k$ as if from a fresh start and ignoring the gathered information $\psi_i$. 

Moreover, due to the adaptive monotonicity of $f$, by \cite[Lemma A.8]{golovin2011}, we have
$f_{avg}(\pi^*_k) \leq f_{avg}(\pi_i@\pi^*_k).$ 
Putting them together, we have for all $i=0,\hdots,\ell$,
\begin{align*}
f_{avg}(\pi^*_k)-f_{avg}(\pi_i)  \leq  \zeta k (f_{avg}(\pi_{i+1})-f_{avg}(\pi_i)).
\end{align*}
Then, defining $\Delta_i \triangleq f_{avg}(\pi^*_k)-f_{avg}(\pi_i)$, we obtain $\Delta_{i+1} \leq (1-\frac{1}{ \zeta k}) \Delta_i$ and thus,  $\Delta_{\ell} \leq (1-\frac{1}{ \zeta k})^\ell \Delta_0 < e^{-\ell/\zeta k} \Delta_0$, where we used the fact that $1-x < e^{-x}$ for all $x$ to obtain the last inequality. Finally, by the definition of $\Delta_i$, 
\begin{align*}
f_{avg}(\pi^*_k)-f_{avg}(\pi_\ell)&<e^{-\ell/\zeta k}(f_{avg}(\pi^*_k)-f_{avg}(\pi_0))\\
&\leq e^{-\ell/\zeta k}f_{avg}(\pi^*_k),
\end{align*}
since $f_{avg}(\pi_0)\geq 0$ by assumption. Hence, the result in the theorem follows directly from the above.
\end{proof}

\subsection{Proofs for Near-Optimal Group-Based Active Diagnosis} \label{app:proofs}
\subsubsection{Conditional Expected Marginal Benefit $\Delta (v|\psi_t)$}
Our proofs for Propositions \ref{prop:monotone} and \ref{prop:submodular} rely heavily on the particular form that $\Delta(v|\psi_t)$ takes on, given in the following:
\begin{lem}[Conditional Expected Marginal Benefit] \label{lem:Delta}
The conditional expected marginal benefit of action $v$ for the reward function $f(v_{1:t},x,q)$ in \eqref{eq:reward} can be expressed as
\begin{align} \label{eq:deltasup}
\Delta (v|\psi_t)= \sum_{i=1}^{|\mathcal{Y}|} \zeta_{i} \tau_{i} -\frac{\sum_{i=1}^{|\mathcal{Y}|} \zeta_{i} \tau_{i}^2}{\sum_{i=1}^{|\mathcal{Y}|} \tau_{i}},
\end{align}
for some $\tau_{i} \geq 0$ and $ \zeta_{i} \geq 1$ that are both explicitly dependent on the partial realization $\psi_t$ for all $1 \leq i \leq |\mathcal{Y}|$.
\end{lem}

\begin{proof}
By definition, 
\begin{align*}
\Delta(v|\psi_t)&=\mathbb{E}[f(v_{1:t}\cup \{v\},X,Q)-f(v_{1:t},X,Q)| \psi_t]\\
&=\mathbb{E}[\tilde{f}(v_{1:t}\cup \{v\},X,Q)-\tilde{f}(v_{1:t},X,Q)| \psi_t]
\end{align*} with the expectation taken with respect to $\mathbb{P}[x,q|\psi_t]$ and where $\tilde{f}(v_{1:t},x,q) \triangleq -\sum_{x \in \bigcup\limits_{q \in \mathcal{Q}} S_{t,q}} \mathbb{P}[x]$. The individual terms can be computed as:
\begin{align}
&\hspace{-0.3cm} \begin{array}{ll}
\mathbb{E}[\tilde{f} (v_{1:t}\cup \{v\},X,Q)| \psi_t]\\
=\sum\limits_{q \in \mathcal{Q}} \sum\limits_{x \in S_{t,q}} \mathbb{P}[x,q|\psi_t] \tilde{f} (v_{1:t}\cup\{v\},x,q)\\
=-\sum\limits_{q \in \mathcal{Q}} \sum\limits_{x \in S_{t,q}} \hspace*{-0.1cm}\mathbb{P}[x,q|\psi_t] \hspace*{-0.1cm}\sum\limits_{\tilde{x} \in \bigcup\limits_{\tilde{q} \in \mathcal{Q}} (S_{t,\tilde{q}} \cap D(\mu(v,x,q),v,\tilde{q}))} \hspace*{-0.1cm}\mathbb{P}[\tilde{x}], \hspace*{-0.35cm}
\end{array} \label{eq:Stp1}\\
%\end{align}
%\begin{align}
&\hspace{-0.3cm} \begin{array}{ll}
\mathbb{E}[\tilde{f} (v_{1:t},X,Q)| \psi_t]\hspace*{-0.3cm}&=\sum\limits_{q \in \mathcal{Q}} \sum\limits_{x \in S_{t,q}} \mathbb{P}[x,q|\psi_t] \tilde{f} (v_{1:t},x,q)\\
\hspace*{-0.3cm}&=-\sum_{\tilde{x} \in \bigcup\limits_{\tilde{q} \in \mathcal{Q}} S_{t,\tilde{q}}} \mathbb{P}[\tilde{x}],
\end{array} \hspace*{-0.3cm}\label{eq:St1}
\end{align}
where the final equality holds because $\tilde{f} (v_{1:t},x,q)$ is independent of $q$ and $x$.
%where only the final terms in \eqref{eq:Stp1} and \eqref{eq:St} contribute to $\Delta (v|\psi_t)$, as was remarked in Section \ref{sec:math}. 

To proceed, since the probability measure on $(\mathcal{X},\mathcal{Q})$ is non-uniform and can take values in some set $\{p_1,\hdots,p_N\}$, we define subsets of $(\mathcal{X},\mathcal{Q})$ where $\mathbb{P}[x,q]$ is constant with $F_{n,q}\triangleq\{(x,q) \in \mathcal{X} \times \mathcal{Q}\ |\ \mathbb{P}[x,q] =p_n\}$. Furthermore, since for every $x \in \mathcal{X}$, there is a corresponding $y=\mu(v,x,q)$ for any given  $v\in \mathcal{V}$ and $q \in \mathcal{Q}$, the sets $\{D(y,v,q)\ |\ y \in \mathcal{Y}\}$ also form a partition of $\mathcal{X}$. Finally, we define a partition of $S_{t,q}$ for each $q \in \mathcal{Q}$ using the sets $\{S_{t,q} \cap D(y,v,q) \cap F_{n,q}\ | \ y \in \mathcal{Y}, n \in 1:N\}$ and denote with a shorthand $\alpha_{n,y,q} \triangleq S_{t,q} \cap D(y,v,q) \cap F_{n,q}$. Thus, for all $(x,q) \in \alpha_{n,y,q}$, $\mu(v,x,q)=y$ and $\mathbb{P}[x,q]=p_n$. %, and we denote the cardinality of the set $\alpha_{n,y}$ as $|\alpha_{n,y}|$.

Moreover, for all $x \in F_{n,q}$, the conditional probabilities in \eqref{eq:condP} can be rewritten as
\begin{align}
\mathbb{P}[{x,q}|\psi_t]=\frac{p_n}{\mathbb{P}[\psi_t]} \triangleq \frac{p_n}{\sum_{y \in \mathcal{Y}} \tau_{y}}, \label{eq:condPR}
\end{align}
where we obtain the expression for $\mathbb{P}[\psi_t]$ from \eqref{eq:normalization} as
\begin{align*}
\begin{array}{rl}
\mathbb{P}[\psi_t]&=\sum\limits_{q\in \mathcal{Q}}\sum\limits_{x \in S_{t,q}} \mathbb{P}[x,q]\\
&=\sum\limits_{n \in 1:N} \sum\limits_{y \in \mathcal{Y}} \sum\limits_{q\in \mathcal{Q}} \sum_{x \in S_{t,q} \cap D(y,v,q) \cap F_{n,q}} p_n\\
&=\sum\limits_{y\in \mathcal{Y}} \sum\limits _{n \in 1:N}  \sum\limits_{q\in \mathcal{Q}}  p_n |\alpha_{n,y,q}|
=\sum\limits_{y \in \mathcal{Y}} \tau_{y},
\end{array}
\end{align*}
and defined $\tau_{y} \triangleq \displaystyle\sum_{n \in 1:N}  \sum_{q\in \mathcal{Q}}  p_n |\alpha_{n,y,q}| \geq 0$.

Furthermore, the following inequality holds:
\begin{align}
\sum\limits_{{x} \in \bigcup\limits_{\tilde{q} \in \mathcal{Q}} \hspace{-0.1cm} S_{t,\tilde{q}}} \hspace{-0.25cm}\mathbb{P}[{x}] = \sum\limits_{q \in \mathcal{Q}} \sum\limits_{{x} \in \bigcup\limits_{\tilde{q} \in \mathcal{Q}} \hspace{-0.1cm} S_{t,\tilde{q}}} \hspace{-0.25cm} \mathbb{P}[x,q]\geq \sum\limits_{q \in \mathcal{Q}} \sum_{x \in S_{t,q}}\hspace{-0.1cm} \mathbb{P}[x,q], \hspace{-0.1cm}\label{eq:inEqZeta} 
\end{align}
since on the right hand side, $q$ is a particular choice of $\tilde{q} \in \mathcal{Q}$. Using the above-defined partitions, we can equivalently write
\begin{align*}
\begin{array}{rl}
\sum\limits_{{x} \in \bigcup\limits_{\tilde{q} \in \mathcal{Q}} S_{t,\tilde{q}}\cap D(y,v,\tilde{q})} \hspace*{-0.2cm}\mathbb{P}[{x}] &=\sum\limits_{n \in 1:N} \sum\limits_{{x} \in \bigcup\limits_{\tilde{q} \in \mathcal{Q}} S_{t,\tilde{q}}\cap D(y,v,\tilde{q}) \cap F_{n,\tilde{q}}}\hspace*{-0.2cm} \mathbb{P}[{x}] \\
&= \zeta(\psi_t,y) \sum\limits_{n \in 1:N} \sum\limits_{q \in \mathcal{Q}} \sum_{x \in \alpha_{n,y,q}} p_n\\ %\sum_{x \in S_{t,q} \cap D(y,v,q) \cap F_{n,q}} p_n, %\label{eq:ineq}
&=\zeta(\psi_t,y)\tau_y,
\end{array}
\end{align*}
where we defined
\begin{align}
\zeta(\psi_t,y)\triangleq \frac{\sum_{{x} \in \bigcup\limits_{\tilde{q} \in \mathcal{Q}} S_{t,\tilde{q}}\cap D(y,v,\tilde{q})} \mathbb{P}[{x}]}{\sum_{q \in \mathcal{Q}} \sum_{{x} \in  S_{t,q}\cap D(y,v,q)}  \mathbb{P}[{x,q}]}  \label{eq:zetaDef}
\end{align}
and from \eqref{eq:inEqZeta}, we have $\zeta(\psi_{t},y) \geq 1$. %which depends on $\psi'_{t}\triangleq \psi_{t}\cup \{v,\mu(v,x)\}$, which 
This simplifies \eqref{eq:Stp1} to
\begin{align}
\begin{array}{ll}
\mathbb{E}[\tilde{f} (v_{1:t}\cup \{v\},X,Q)| \psi_t]\\
%=-1 - \sum_{x \in S_t} \mathbb{P}[x|\psi_t] \sum_{R \in \mathcal{R}} \mathbb{P}[R: x \in R]\\
=-\sum\limits_{q \in \mathcal{Q}} \sum\limits_{x \in S_{t,q}} \hspace*{-0.1cm}\mathbb{P}[x,q|\psi_t]  \zeta(\psi_t,\mu(v,x,q)) \tau_{\mu(v,x,q)}\\
=-\frac{1}{\mathbb{P}[\psi_t]}\sum\limits_{y\in \mathcal{Y}}\sum \limits_{n \in 1:N}  \sum\limits_{q \in \mathcal{Q}} \sum\limits_{x \in \alpha_{n,y,q}} \hspace*{-0.1cm}\mathbb{P}[x,q]  \zeta(\psi_t,y) \tau_{y}\\
=-\frac{\sum_{y\in \mathcal{Y}} \zeta(\psi_t,y) \tau^2_{y}}{\sum_{y\in \mathcal{Y}} \tau_{y}}.
\end{array} \label{eq:Stp}
\end{align}

%Using these partitions and the conditional probabilities in \eqref{eq:condPR}, the final term in \eqref{eq:Stp} becomes
%\begin{align}
%\begin{array}{l}
%\displaystyle\sum_{x \in S_t} \mathbb{P}[x|\psi_t]  \sum_{\tilde{x} \in S_t \cap D(\mu(v,x),v)} \kappa_R(\tilde{x},\psi'_t) \mathbb{P}[\tilde{x}]\\
%=\displaystyle\sum_{x \in S_t} \mathbb{P}[x|\psi_t]  \sum_{n \in 1:N} \sum_{\tilde{x} \in \alpha_{n,y}} \kappa_R(\tilde{x},\psi'_t) p_n\\
%=\displaystyle\sum_{x \in S_t} \mathbb{P}[x|\psi_t] \zeta_R(\psi'_t)  \sum_{n \in 1:N} \sum_{\tilde{x} \in \alpha_{n,y}}  p_n\\
%=\displaystyle\sum_{x \in S_t} \mathbb{P}[x|\psi_t] \zeta_R(\psi'_t) \tau_{R,\mu(v,x)}\\
%=\displaystyle\frac{1}{\sum_{y \in \mathcal{Y}} \tau_{R,y}} \displaystyle\sum_{n \in 1:N} \sum_{y \in \mathcal{Y}} \sum_{x \in \alpha_{n,y}} p_n \zeta_R(y) \tau_{R,y}\\
%=\displaystyle\frac{1}{\sum_{y \in \mathcal{Y}} \tau_{R,y}}\displaystyle \sum_{y \in \mathcal{Y}}  \zeta_R(y) \tau_{R,y} \sum_{n \in 1:N} \sum_{x \in \alpha_{n,y}} p_n \\
%=\displaystyle\frac{1}{\sum_{y \in \mathcal{Y}} \tau_{R,y}} \displaystyle \sum_{y \in \mathcal{Y}}  \zeta_R(y) \tau_{R,y}^2
%\end{array} \label{eq:Stpf}
%\end{align}
%where once again, we used $\alpha_{n,y} \triangleq S_t \cap D(y,v) \cap F_n$ and $$\sum_{n \in 1:N} \sum_{\tilde{x} \in \alpha_{n,y}} \kappa_R(\tilde{x},\psi'_t) p_n=\zeta_R(\psi'_t)  \sum_{n \in 1:N} \sum_{\tilde{x} \in \alpha_{n,y}}  p_n$$ for $\kappa(\tilde{x},\psi'_t) \geq 1$ implies that $\zeta_R(\psi'_t) \geq 1$.

Similarly, \eqref{eq:St1} simplifies to
\begin{align}
\hspace{-0.3cm} \begin{array}{ll}
\mathbb{E}[\tilde{f} (v_{1:t},X,Q)| \psi_t]\hspace*{-0.3cm}&=-\sum\limits_{y \in \mathcal{Y}}\sum\limits_{\tilde{x} \in \bigcup\limits_{\tilde{q} \in \mathcal{Q}} S_{t,\tilde{q}}\cap D(y,v,\tilde{q})} \mathbb{P}[\tilde{x}]\\
\hspace*{-0.3cm}&=-\sum_{y \in \mathcal{Y}} \zeta(\psi_t,y)\tau_y.
\end{array} \label{eq:St}
\end{align}
Finally, from \eqref{eq:Stp} and \eqref{eq:St}, we obtain $\Delta (v|\psi_t)$ of the form \eqref{eq:deltasup} with $\zeta_{y}=\zeta(\psi_t,y)\geq 1$ and $\tau_{y}\geq 0$ for all $y \in \mathcal{Y}$.
\end{proof}

\subsubsection{Performance Guarantees} 

Armed with Lemma \ref{lem:Delta}, Proposition \ref{prop:monotone} can be shown to directly hold.
\begin{proof}[Proof of Proposition \ref{prop:monotone}]
From \eqref{eq:deltasup} in Lemma \ref{lem:Delta},
\begin{align*}
\begin{array}{rl}
\Delta(v|\psi_t) \hspace{-0.25cm} &=\displaystyle \frac{ (\sum_{i=1}^{|\mathcal{Y}|} \tau_{i}) (\sum_{i=1}^{|\mathcal{Y}|} \zeta_{i} \tau_{i}) -\sum_{i=1}^{|\mathcal{Y}|} \zeta_{i} \tau_{i}^2}{\sum_{i=1}^{|\mathcal{Y}|} \tau_{i}}\\
&=\displaystyle\frac{ \sum_{i=1}^{|\mathcal{Y}|} (\zeta_{i} \tau_{i} \sum_{j\neq i} \tau_{j}) }{\sum_{i=1}^{|\mathcal{Y}|} \tau_{i}} \geq 0,
\end{array}
\end{align*}
since $\zeta_i \geq 1$ and $\tau_i \geq 0$ for all $i=1,\hdots,|\mathcal{Y}|$. Hence, adaptive monotonicity holds by Definition \ref{def:monotone}.
\end{proof}

To prove Proposition \ref{prop:submodular}, i.e., that the reward function $f$ in \eqref{eq:reward} is weakly adaptive submodular, we first state the following lemma from \cite{maillet2013} that will be useful.
\begin{lem}{\cite[Lemma 1]{maillet2013}} \label{lem:b}
The function $b:\mathbb{R}^{|\mathcal{Y}|} \to \mathbb{R}$
\begin{align} \label{eq:b}
b(\tau_1,\tau_2,\hdots,\tau_{|\mathcal{Y}|})=\sum_{i=1}^{|\mathcal{Y}|} \tau_i -\frac{\sum_{i=1}^{|\mathcal{Y}|}  \tau_i^2}{\sum_{i=1}^{|\mathcal{Y}|}  \tau_i}
\end{align}
is increasing on the positive orthant, i.e., $b(\tau_1,\tau_2,\hdots,\tau_{|\mathcal{Y}|})$ $\geq b(s_1,s_2,\hdots,s_{|\mathcal{Y}|})$ if $\tau _i \geq s_i \geq 0$ for all $1 \leq i \leq |\mathcal{Y}|$.
%Moreover, $b(\tau_1,\tau_2,\hdots,\tau_{|\mathcal{Y}|}) \geq 0$.
\end{lem}
%\begin{proof}
%Since $b$ is permutation invariant with respect to its arguments, it is sufficient to show that it is increasing in one of its arguments. Let $k_1\triangleq \sum_{i=2}^{|\mathcal{Y}|} \zeta_i \tau_i$, $k_2\triangleq \sum_{i=2}^{|\mathcal{Y}|} \zeta_i \overline{\zeta}_i \tau_i^2$  and $k_3\triangleq \sum_{i=2}^{|\mathcal{Y}|} \overline{\zeta}_i \tau_i$. Further, we define $\tilde{b}(x) \triangleq b(x,\tau_2,\hdots,\tau_{|\mathcal{Y}|})=k_1 +\zeta_1 x - \frac{k_2+\zeta_1 \overline{\zeta}_1 x^2}{k_3 + \overline{\zeta}_1 x}$. The partial derivative of $\tilde{b}$ with respect to $x$ is $$\frac{\partial \tilde{b}}{\partial x}=\frac{\zeta_1 k_3^2 + k_2 \overline{\zeta}_1}{(k_3+\overline{\zeta}_1 x)^2},$$
%which is non-negative by the definitions of $k_2$ and $k_3$ and the non-negativity assumption of $\zeta_1$ and $\overline{\zeta}_1$. Finally, $b(\tau_1,\tau_2,\hdots,\tau_{|\mathcal{Y}|})$ is non-negative because $(\sum_{i=1}^{|\mathcal{Y}|} \zeta_i \tau_i) (\sum_{i=1}^{|\mathcal{Y}|} \overline{\zeta}_i \tau_i) \geq \sum_{i=1}^{|\mathcal{Y}|} \zeta_i \overline{\zeta}_i \tau_i^2$ for $\tau_i \geq 0$, $\zeta_i \geq 0$ and $\overline{\zeta}_i \geq 0$ for all $1\leq i \leq |\mathcal{Y}|$.
%\end{proof}

\begin{proof}[Proof of Proposition \ref{prop:submodular}] First, taking the partial derivative of \eqref{eq:deltasup} with respect to $\zeta_i$, we obtain
\begin{align*}
\frac{\partial\Delta (v|\psi_t)}{\partial \zeta_i}=  \tau_{i} -\frac{\tau_{i}^2}{\sum_{i=1}^{|\mathcal{Y}|} \tau_{i}}=\frac{\tau_{i}\sum_{j=1,\hdots,|\mathcal{Y}|, j \neq i} \tau_{j}}{\sum_{i=1}^{|\mathcal{Y}|} \tau_{i}} \geq 0.
\end{align*}

From the above, we can conclude that
\begin{align}
b(\psi_t) \leq b(\psi_t) \min_i \zeta_i  \leq \Delta(v|\psi_t) \leq b(\psi_t)\max_i \zeta_i \leq \zeta b(\psi_t), \label{eq:inEq}
\end{align}
where we used $b(\psi_t)$ as a shorthand for $b(\tau_1,\tau_2,\hdots,\tau_{|\mathcal{Y}|})$ defined in \eqref{eq:b} since $\psi_t$ is related to $\tau_i$'s via \eqref{eq:deltasup}, and $\zeta$ is given in \eqref{eq:zeta}, which is conservatively chosen to be larger than $\max_i \zeta_i$ across all possible realizations and sequence of actions, while by Lemma \ref{lem:Delta}, we have $\min_i \zeta_i \geq 1$ . 

Moreover, the expression for $\Delta(v|\psi_{t'})$ is similar for the partial realization $\psi_{t'}$ but with different $\tau'_i$ and $\zeta'_i$.  For $\psi_t$ as a subrealization of $\psi_{t'}$, i.e., $\psi_t \subseteq \psi_{t'}$, by construction, $S_{t',q} \subseteq S_{t,q}$ for all $q \in \mathcal{Q}$ and by definition of $\tau_i$ and Lemma \ref{lem:b}, we know that $\tau'_i \leq \tau_i$ for all $i$ and $b(\psi_{t'}) \leq b(\psi_t)$.

Note that the following inequality from \eqref{eq:inEq}
$$\Delta(v|\psi_{t'}) \leq \zeta b(\psi_{t'})$$
also holds for $\psi_{t'}$ since we have chosen $\zeta$ to be the maximum across all $t$. Thus, from the above and from \eqref{eq:inEq}, we find
\begin{align*}
\Delta(v|\psi_t) \geq b(\psi_t) \geq b(\psi_{t'}) \geq \frac{1}{\zeta} \Delta(v|\psi_{t'}),
\end{align*}
from which $\zeta$-weak adaptive submodularity holds by Definition \ref{def:submodular}. Further, we can find the upper bounds on $\zeta$ since
\begin{align*}
\begin{array}{rl}
\sum\limits_{x \in \bigcup_{q \in \mathcal{Q}} S_{t,q}\cap D(y,v,{q})} \mathbb{P}[x] &\leq \sum\limits_{q\in \mathcal{Q}}\sum\limits_{x \in  S_{t,{q}}\cap D(y,v,{q})} \mathbb{P}[x]\\
&\leq \sum\limits_{q\in \mathcal{Q}}\sum\limits_{x \in  \mathcal{X}} \mathbb{P}[x] =|\mathcal{Q}|
\end{array}
\end{align*}
and $\sum\limits_{q\in \mathcal{Q}} \sum\limits_{x\in S_{t,q} \cap D(y,v,q)}\hspace*{-0.2cm} \mathbb{P}[x,q]\geq \min\limits_{\{x \in \mathcal{X}, q \in \mathcal{Q}: \mathbb{P}[x,q]>0\}}\mathbb{P}[x,q]$, while $\zeta\geq 1$ by \eqref{eq:zetaDef}.
\end{proof}

\end{document}